\documentclass[11pt]{article}

\usepackage{a4wide,amsfonts,amsmath,latexsym,amsthm,amssymb,euscript,graphicx,units,mathrsfs}
\usepackage{amsmath}
\usepackage{diagbox}
\usepackage{amssymb}
\usepackage{amsthm}
\usepackage{latexsym}
\usepackage{color}
\usepackage{enumerate}
\usepackage{graphicx}
\usepackage{multirow}
\usepackage[T1]{fontenc}
\usepackage[english]{babel}
\usepackage{geometry}
\DeclareSymbolFont{calletters}{OMS}{cmsy}{m}{n}
\DeclareSymbolFontAlphabet{\mathcal}{calletters}
\usepackage{float}
\newfloat{figure}{h}{lof}
\floatname{figure}{\figurename}
\def\be{\begin{eqnarray}}
\def\Ee{\end{eqnarray}}

\def\b*{\begin{eqnarray*}}
\def\E*{\end{eqnarray*}}

\newcommand{\real}{\mathbb{R}}
\newcommand{\pr}{\mathbb{P}}
\newcommand{\sgn}{\text{sgn}}

\newtheorem{Theorem}{Theorem}[section]

\newtheorem{Proposition}[Theorem]{Proposition}

\newtheorem{Lemma}[Theorem]{Lemma}
\newtheorem{Corollary}[Theorem]{Corollary}
\newtheorem{Remark}[Theorem]{Remark}
\newtheorem{Example}[Theorem]{Example}

\makeatletter \@addtoreset{equation}{section}

\newcommand{\No}[1]{\left\|#1\right\|}     
\newcommand{\abs}[1]{\left|#1\right|}     

\def \E{\mathbb{E}}
\def \F{\mathbb{F}}
\def \H{\mathbb{H}}

\def \N{\mathbb{N}}
\def \P{\mathbb{P}}

\def \R{\mathbb{R}}

\def\Bc{{\cal B}}
\def\Cc{{\cal C}}

\def\Fc{{\cal F}}

\def\esup{{\rm ess \, sup}}

\def\={\;=\;}
\def\.{\;.}

\def\reff#1{{\rm(\ref{#1})}}

\def\1{{\bf 1}}

 \def\normeL2#1{\left\|{#1}\right\|_{L^2}}

\def\E{\mathbb{E}}
\def\P{\mathbb{P}}
\def\h{\mathfrak{H}}
\def\S{\mathbb{S}}
\def\H{\mathbb{H}}

\def\b{\textcolor{blue}}

\author{Thibaut Mastrolia\footnote{Universit\'e Paris-Dauphine, CEREMADE UMR CNRS 7534, Place du Mar\'echal De Lattre De Tassigny, 75775 Paris cedex 16, France, mastrolia@ceremade.dauphine.fr} \and Dylan Possama\"i\footnote{Universit\'e Paris-Dauphine, CEREMADE UMR CNRS 7534, Place du Mar\'echal De Lattre De Tassigny, 75775 Paris cedex 16, France, possamai@ceremade.dauphine.fr} \and Anthony R\'eveillac\footnote{INSA de Toulouse, IMT UMR CNRS 5219, Université de Toulouse, 135 avenue de Rangueil 31077 Toulouse Cedex 4, France, anthony.reveillac@ceremade.dauphine.fr} }

\title{Density analysis of BSDEs}

\begin{document}

\maketitle

\vspace{-2.5em}
\begin{abstract}
In this paper, we study the existence of densities (with respect to the Lebesgue measure) for marginal laws of the solution $(Y,Z)$ to a quadratic growth BSDE. Using the (by now) well-established connection between these equations and their associated semi-linear PDEs, together with the Nourdin-Viens formula, we provide estimates on these densities.  
\end{abstract}

\vspace{0.2em} 
{\noindent \textit{Key words:} BSDEs; Malliavin Calculus; Density analysis; Nourdin-Viens' Formula; PDEs.
}

\vspace{0.2em}
\noindent
{\noindent \textit{AMS 2010 subject classification:} Primary: 60H10; Secondary: 60H07.
\normalsize
}

\vspace{-1em}
\tableofcontents
\section{Introduction}
In recent years the field of Backward Stochastic Differential Equations (BSDEs) has been a subject of growing interest in stochastic calculus, as these equations naturally arise in stochastic control problems in Finance, and as they provide Feynman-Kac type formulae for semi-linear PDEs (\cite{PardouxPeng_92}). Before going further let us recall that a solution to a BSDE is a pair of \textit{regular enough} (in a sense to be made precise) predictable processes $(Y,Z)$ such that
\begin{equation}
\label{eq:BSDEintro}
Y_t=\xi+\int_t^T h(s,Y_s,Z_s)ds -\int_t^T Z_s dW_s, \quad t\in [0,T], 
\end{equation}   
where $W$ is a one-dimensional Brownian motion, $h$ is a predictable process and $\xi$ is a $\mathcal{F}_T$-measurable random variable (with $(\mathcal{F}_t)_{t \in [0,T]}$ the natural completed and right-continuous filtration generated by $W$). Since it is generally not possible to provide an explicit solution to \eqref{eq:BSDEintro}, except for instance when $h$ is a linear mapping of $(y,z)$, one of the main issues, especially regarding the applications is to provide a numerical analysis for the solution of a BSDE. This calls for a deep understanding of the regularity of the solution processes $Y$ and $Z$. The classical regularity related to obtaining a numerical scheme for the solution $(Y,Z)$ is the so-called \textit{path regularity} for the $Z$ component originally studied in \cite{Ma_Zhang_PTRF02}. In this paper, we aim at studying another type of regularity namely, we focus on the law of the marginals of the random variables $Y_t$, $Z_t$ at a given time $t$ in $(0,T)$. More precisely, we are interested in providing sufficient conditions which ensure the existence of a density (with respect to the Lebesgue measure) for these marginals on the one hand, and in deriving some estimates on these densities on the other hand. This type of information on the solution is of theoretical and of practical interest since the description of the tails of the (possible) density of $Z_t$ would provide more accurate estimates on the convergence rates of numerical schemes for quadratic growth BSDEs (qgBSDEs in short), that is when $h$ in \eqref{eq:BSDEintro} has quadratic growth in the $z$-variable, as noted in \cite{DosReisPhD}.

\vspace{0.3em}    
Before reviewing the results available in the literature and the one we derive in this paper, we would like to illustrate with the two following simple examples that the existence and the estimate of densities issues for BSDEs are very different from the one concerning the classical (forward) SDEs. For instance consider the following very particular case of \eqref{eq:BSDEintro} given by:
\begin{equation}
\label{eq:BSDEintrobis}
Y_t=W_1+\int_t^T (s-W_s) ds -\int_t^T Z_s dW_s, \quad t\in [0,1], \; (T=1). 
\end{equation}   
This equation should be extremely simple in the sense that the driver $h$ does not depend on $(Y,Z)$, and indeed it can be solved explicitly to get that:
$$ Y_t = W_t \left(-\frac12 + 2t -\frac{t^2}{2}\right), \quad t\in [0,1]. $$
Hence $Y_t$ is a Gaussian random variable for every time $t$ in $(0,2-\sqrt{3})$, then $Y_{2-\sqrt{3}} = 0$, and for $t$ in $(2-\sqrt{3},1]$, $Y_t$ is Gaussian distributed once again. This illustrates the difficulty of the problem and somehow shows how it is different from the study of forward SDEs. This example, even though it is very simple is pretty insightful and will be studied as Example \ref{exemple} in Section \ref{section:lip}. 
%
Concerning the density estimates, the backward case brings here also, significant differences with the forward case as the following example illustrates. Consider the following equation:
\begin{equation}
\label{eq:BSDEintroter}
Y_t=W_1^3+\int_t^T 3 W_s ds -\int_t^T Z_s dW_s, \quad t\in [0,1], \; (T=1),
\end{equation}   
which can be solved explicitly:
 $$ Y_t = W_t^3 +6 W_t(1-t), \quad Z_t = 3 W_t^2 + 6(1-t), \; t\in [0,1],$$
from which we deduce that both $Y_t$ and $Z_t$ admits a density with respect to the Lebesgue's measure for $t$ in $(0,1]$. However, it is clear that neither the law of $Y_t$ nor the one of $Z_t$ admits Gaussian tails. This example will be considered in Section \ref{section:densY} as Example \ref{rem.toostringent}.

\vspace{0.3em}    
Coming back to the 
general problem of existence of densities for the marginal laws of $Y$ and $Z$, it is worth mentioning that this issue has been pretty few studied in the literature, since up to our knowledge only references \cite{AntonelliKohatsu,AbouraBourguin} address this question. The first results about this problem have been derived in \cite{AntonelliKohatsu}, where the authors provide existence and smoothness properties of densities for the marginals of the $Y$ component only and when the driver $h$ is Lipschitz continuous in $(y,z)$. Note that two kinds of sufficient conditions for the existence of a density for $Y$ are derived in \cite{AntonelliKohatsu}: the so-called \textit{first-order} (cf. \cite[Theorem 3.1]{AntonelliKohatsu}) and \textit{second-order} (see \cite[Theorem 3.6]{AntonelliKohatsu}) conditions. Concerning the $Z$ component, much less is known since existence of a density for $Z$ has been established in \cite{AbouraBourguin} only under the condition that the driver is linear in $z$. This constitutes, to our point of view, a major restriction since up to a Girsanov transformation this case basically reduces to the situation where the driver does not depend on $z$. Nonetheless, in \cite{AbouraBourguin}, estimates on the densities of the laws of $Y_t$ and $Z_t$ are given using the Nourdin-Viens formula.

\vspace{0.3em}
In this paper we revisit and extend the results of \cite{AntonelliKohatsu,AbouraBourguin} by providing sufficient conditions for the existence of densities for the marginal laws of the solution $Y_t,Z_t$ (with $t$ an arbitrary time in $(0,T)$) of a qgBSDE with a terminal condition $\xi$ in \eqref{eq:BSDEintro} given as a deterministic mapping of the value at time $T$ of the solution to a one-dimensional SDE, together with some estimates on these densities. The results concerning the Lipschitz case, \textit{i.e.} when the generator $h$ is Lipschitz, are presented in Section \ref{section:lip}. As recalled above, the case where $h$ is Lipschitz continuous in $(y,z)$ has been investigated in \cite{AntonelliKohatsu} for the $Y$ component only, where the authors have derived two types of sufficient conditions. However, we provide as Example \ref{exemple} a counter-example to \cite[Theorem 3.6]{AntonelliKohatsu} which is devoted to the second-order conditions. This is due to an inefficiency in the proof that can be easily fixed by making a small change in a key quantity in the statement of the result. Hence, we propose a new version of this result as Theorem \ref{AKmodifie}. Then, we gather in Section \ref{section:lip:z} the first existence results of a density for the $Z$ component for Lipschitz BSDEs. Concerning the quadratic case, studied in Section \ref{section:quadratic}, we propose sufficient conditions for the existence of a density first for the $Y$ component of qgBSDEs (in Section \ref{section:quadratic:y}), then for the $Z$ component of qgBSDEs (in Section \ref{section:quadratic:z}). We would like to stress once more at this stage that concerning the existence of a density for the $Y$ component, only the Lipschitz case was known and concerning the control variable $Z$, only the case of linear drivers in $z$ was studied (see \cite[Theorem 4.3]{AbouraBourguin}) up to now, which makes our result a major improvement on the existing literature. Finally, we derive in Section \ref{section:densY}, density estimates for the marginal laws of $Y$ and $Z$ using the Nourdin-Viens formula, and taking advantage of the connection between the solution to a Markovian BSDE and the solution to its associated semi-linear PDE. Note that contrary to \cite{AbouraBourguin}, we do not assume that the Malliavin derivative of $Y$ (or $Z$) to be bounded which is, from our point of view, a too stringent assumption (as illustrated in Example \ref{rem.toostringent}) both from the theoretical and practical point of view. Indeed, such an assumption leads to Gaussian tails for the densities of $Y$ or $Z$. 
However, even in seemingly benign situations, we will see that it is not generally the case for BSDEs, and unlike most of the literature, we have obtained tail estimates which are {\it not} Gaussian. This might be seen as a significant difference between BSDEs and diffusive equations (i.e. with an initial condition) like SDEs or SPDEs for instance \cite{Kohatsu_03,Kohatsu_2003b,Nualart_Quer}.
\vspace{0.5em}

\noindent
Before going further, we would like to explain why our results are quite relevant for financial applications and some stochastic control problems. Most of problems in portfolio management, utility maximization or risk sensitive control (see \textit{e.g.} \cite[Section 4.2]{elk_hamadene_matoussi}) can be essentially reduced to study a qgBSDE. Let us present two examples.
\begin{itemize}
\item[1.] Assume that a financial agent wants to maximize her utility under constraints, \textit{i.e.} her investment strategies are restricted to a specific closed set $C$, it was proved in \cite{elk_rouge} and \cite{HIM} that her optimal strategies are essentially given through the $Z$ component of a qgBSDE of the form
$$Y_t=\xi +\int_t^T h(s,Z_s)ds -\int_t^T Z_s dW_s, \; \forall t\in [0,T]\; \mathbb{P}-a.s.$$
with
$$h(s,z):= -z\theta_s - \frac{|\theta_s|^2}{2\alpha}+\frac\alpha2 \text{dist}_{C}^2\left(z+\frac{\theta_s}{\alpha}\right),$$  where $\alpha$ denotes the risk aversion of the investor and $\theta$ is the market price of risk, and where $\text{dist}_C(x)$ denotes the Euclidean distance between $x$ and $C$. Hence, if one obtains a criterion providing density existence for the $Z$ component solution to a qgBSDE with estimates on its tails, then one gets crucial information to study the behaviors of optimal strategies for utility maximization problems. For example, since $Z$ essentially gives the optimal quantity of money which should be invested in the risky asset, being able to estimate the probability that $Z$ becomes large is particularly meaningful in risk management. Besides, the control of the tails of the density of $Z$ could give important information concerning the rate of convergence for numerical schemes to solve numerically BSDEs, so as to compute optimal strategies (see \cite{ImkellerDosreis, ChassagneuxRichou}). For instance, one can check directly that if $\theta$ above is deterministic, $C$ is smooth (that is its boundary is a $C^2$ Jordan arc), and $\xi=g(W_T)$, where $g$ is any bounded function such that its second-order derivative is non-negative almost everywhere and positive on a set of positive Lebesgue measure (for instance a smoothed butterfly spread), then Theorem \ref{thm_density_z_quadra} below applies and $Z_t$ admits a density for all $t\in(0,T]$.

\item[2.] Assume now that a controller, sensitive to risk, wants to maximize on the control set $ \mathcal{U}$
\begin{equation}\label{RSP} J(u):= \mathbb{E}^u\left[ \exp\left(\theta \int_0^T H(s,X_{\cdot},u_s)ds + g(X_T)\right)\right], \; u \in \mathcal{U},\end{equation}
where $\theta$ denotes the sensitiveness of the controller with respect to risk and $X$ denotes a solution to a classical SDE. This is the classical risk sensitive control problem introduced in \cite{jacobson}. Hence, this risk sensitive control problem can be rewritten in term of the well-known risk entropic measure (see \cite{barrieu_elk} for more details). Then, according to \cite[Theorem 4.3]{elk_hamadene_matoussi}, one can find a maximizer $u^\star$ of \eqref{RSP} which is essentially given by a process $Z^\star$ which is the second component of the solution to the following qgBSDE
$$ Y_t^\star= g(X_T)+\int_t^T h(s,x_{\cdot}, Z_s^\star, u^\star_s)+\frac12 |Z^\star_s|^2ds -\int_t^T Z_s^\star dW_s, \;\forall t\in [0,T], \; \mathbb{P}-a.s.,$$

where $h$ is the Hamiltonian process (which is given explicitly in terms of $H$), which is such that $z\longmapsto h(s,x_{\cdot}, z, u_s)+\frac12 |z|^2$ has a quadratic growth for every $s\in [0,T]$ and $u\in \mathcal{U}$. Again, our results give information on the density of $Z^\star$ and thus on the law of the optimal control which is important for obtaining qualitative properties of this optimal control as well as for numerical approximations.
\end{itemize}
\section{Preliminaries}
\subsection{General notations}

In this paper we fix $T \in (0,\infty)$. Let $W:=(W_t)_{t\in [0,T]}$ be a standard one-dimensional Brownian motion on a probability space $(\Omega,\mathcal{F},\P)$, and we denote by $\mathbb{F}:=(\mathcal{F}_t)_{t\in [0,T]}$ the natural (completed and right-continuous) filtration generated by $W$. We denote by $\lambda$ the Lebesgue measure on $\real$ and we set for any $p\in [1,+\infty]$,  $L^p(\P):=L^p(\Omega,\mathcal{F}_T,\P)$ and denote by $\No{\cdot}_p$ the associated norm.
\vspace{0.3em}
We denote by $\mathcal{C}_b(\real^n)$ ($n \geq 1$) the set of functions from $\real^n$ to $\real$ which are infinitely differentiable with bounded partial derivatives. Similarly, for any $n\geq 1$ and any $p\in\N^*$, we denote by $\mathcal C^p(\R^n)$ the set of functions $f:\R^n\rightarrow \R$ which are $p$-times continuously differentiable. For $f$ in $\mathcal{C}_b(\real^n)$, we set $f_{x_{i_1}\cdots x_{i_n}}$ the $n$-th partial derivative with respect to the variables $x_{i_1},\ldots,x_{i_k}$ with $i_1+\ldots+i_k=n$. For a differentiable mapping $f:\real \longrightarrow \real$, we denote $f'$ its derivative in place of $f_x$. Let us denote, for any $(p,q)\in\N^2$, by $\Cc^{p,q}$ the space of functions $f:[0,T]\times\R\rightarrow \R$ which are $p$-times differentiable in $t$ and $q$-times differentiable in space with partial derivatives continuous (in $(t,x)$).

\vspace{0.3em}
Finally, we introduce the following norms and spaces for any $p\geq 1$. $\mathbb S^p$ is the space of $\mathbb R$-valued, continuous and $\mathbb F$-progressively measurable processes $Y$ s.t.
$$\No{Y}^p_{\mathbb S^p}:=\mathbb E\left[\underset{0\leq t\leq T}{\sup} |Y_t|^p\right]<+\infty.$$

$\mathbb S^\infty$ is the space of $\mathbb R$-valued, continuous and $\F$-progressively measurable processes $Y$ s.t.
$$\No{Y}_{\mathbb S^\infty}:=\underset{0\leq t\leq T}{\sup}\No{Y_t}_\infty<+\infty.$$

$\mathbb H^p$ is the space of $\mathbb R$-valued and $\F$-predictable processes $Z$ such that
$$\No{Z}^p_{\mathbb H^p}:=\mathbb E\left[\left(\int_0^T\abs{Z_t}^2dt\right)^{\frac p2}\right]<+\infty.$$

$\rm{BMO}$ is the space of square integrable, continuous, $\mathbb R$-valued martingales $M$ such that
$$\No{M}_{\rm{BMO}}:=\underset{\tau\in\mathcal T_0^T}{\esup}\No{\mathbb E_\tau\left[\left(M_T-M_{\tau}\right)^2\right]}_{\infty}<+\infty,$$
where for any $t\in[0,T]$, $\mathcal T_t^T$ is the set of $\F$-stopping times taking their values in $[t,T]$. Accordingly, $\mathbb H^2_{\rm{BMO}}$ is the space of $\mathbb R$-valued and $\F$-predictable processes $Z$ such that
$$\No{Z}^2_{\mathbb H^2_{\rm{BMO}}}:=\No{\int_0^.Z_sdB_s}_{\rm{BMO}}<+\infty.$$

\subsection{Elements of Malliavin calculus and density analysis}

In this section we introduce the basic material on the Malliavin calculus that we will use in this paper. Set $\h:=L^2([0,T],\mathcal B([0,T]),\lambda)$, where $\mathcal B([0,T])$ is the Borel $\sigma$-algebra on $[0,T]$, and let us consider the following inner product on $\h$
$$\langle f,g\rangle :=\int_0^Tf(t)g(t)dt, \quad \forall (f,g) \in \h^2,$$
with associated norm $\No{\cdot}_{\h}$. Let $\mathcal{S}$ be the set of cylindrical functionals, that is the set of random variables $F$ in $L^2(\P)$ of the form
\begin{equation}
\label{eq:cylindrical}
F=f(W_{t_1},\ldots,W_{t_n}), \quad (t_1,\ldots,t_n) \in [0,T]^n, \; f \in \mathcal{C}_b(\real^n), \; n\geq 1.
\end{equation}
For any $F$ in $\mathcal{S}$ of the form \eqref{eq:cylindrical}, the Malliavin derivative $D F$ of $F$ is defined as the following $\h$-valued random variable:
\begin{equation}
\label{eq:DF}
D F:=\sum_{i=1}^n f_{x_i}(W_{t_1},\ldots,W_{t_n}) \textbf{1}_{[0,t_i]}.
\end{equation}
It is then customary to identify $DF$ with the stochastic process $(D_t F)_{t\in [0,T]}$. Denote then by $\mathbb{D}^{1,2}$ the closure of $\mathcal{S}$ with respect to the Sobolev norm $\|\cdot\|_{1,2}$, defined as:
$$ \|F\|_{1,2}:=\E\left[|F|^2\right] + \E\left[\int_0^T |D_t F|^2 dt\right]. $$
In an iterative way, one may define $D^n F$ (for $n\geq 1$) as the following $\h^{\odot n}$-valued random variable:
$$ D^nF:=D (D^{n-1} F), $$
where $\h^{\odot n}$ denotes the $n$-times symmetric tensor product of $\h$. We refer to \cite{Nualartbook} for more details.

\vspace{0.3em}
We recall the following criterion for absolute continuity of the law of a random variable $F$ with respect to the Lebesgue measure.

\begin{Theorem}[Bouleau-Hirsch, see e.g. Theorem 2.1.2 in \cite{Nualartbook}]\label{BH}
Let $F$ be in $\mathbb{D}^{1,2}$. Assume that $\|DF\|_{\h} >0$, $\mathbb{P}-$a.s. Then $F$ has a probability distribution which is absolutely continuous with respect to the Lebesgue measure on $\mathbb{R}$.
\end{Theorem}

Let $F$ such that $\|DF\|_{\h} >0$, $\mathbb{P}-$a.s., then the previous criterion implies that $F$ admits a density $\rho_F$ with respect to the Lebesgue measure. Assume there exists in addition a measurable mapping $\Phi_F$ with $\Phi_F : \mathbb{R}^{\h}  \rightarrow \h$, such that $DF=\Phi_F(W)$, then we set:
\begin{equation}\label{gzt}
g_F(x):=\int_0^\infty e^{-u} \mathbb{E}\left[\mathbb{E}^*[\langle \Phi_F(W),\widetilde{\Phi_F^u}(W)\rangle_{\h}] \Big{|} F-\mathbb{E}(F)=x\right] du, \ x\in \real,
\end{equation} 
where $\widetilde{\Phi_F^u}(W):=\Phi_F(e^{-u}W+\sqrt{1-e^{-2u}}W^*)$ with $W^*$ an independent copy of $W$ defined on a probability space $(\Omega^*,\mathcal{F}^*,\mathbb{P}^*)$, and $\mathbb{E}^*$ denotes the expectation under $\mathbb{P}^*$ ($\Phi_F$ being extended on $\Omega\times \Omega^*$). We recall the following result from \cite{NourdinViens}.

\begin{Theorem}[Nourdin-Viens' formula]\label{thm_NourdinViens}
$F$ has a density $\rho$ with the respect to the Lebesgue measure if and only if the random variable $g_F(F-\mathbb{E}[F])$ is positive a.s.. In this case, the support of $\rho$, denoted by $\text{supp}(\rho)$, is a closed interval of $\mathbb{R}$ and for all $x \in \text{supp}(\rho)$:
\begin{equation*}
\rho(x)=\frac{\mathbb{E}(|F-\mathbb{E}[F]|)}{2g_{F}(x-\E[F])}\exp{\left( -\int_0^{x-\mathbb{E}[F]} \frac{udu}{g_F(u)} \right)}.
\end{equation*}
\end{Theorem}
\subsection{The FBSDE under consideration}
\label{sub:X}

In this paper, we consider a FBSDE of the form: 
\begin{equation}\label{edsr}
\begin{cases}
\displaystyle X_t= X_0+\int_0^t b(s,X_s) ds +\int_0^t \sigma(s,X_s)dW_s,\ t\in[0,T],\ \P-a.s.\\
\displaystyle Y_t = g(X_T) +\int_t^T h (s,X_s,Y_s,Z_s) ds -\int_t^T Z_s dW_s, \ t\in [0,T],\ \P-a.s.,
\end{cases} 
\end{equation} with $X_0$ a given real constant. We denote by $\mathfrak S(X_t)$ the support of the law of $X_t$ under $\mathbb P$, that is to say the smallest closed subset $A$ of $\real$ such that $\mathbb{P}(X_t\in A)=1$. 
Throughout this paper we will make the following standing assumption on the process $X$ in \eqref{edsr}.

\vspace{0.3em}
\textbf{Standing  assumptions on $X$:} 
\begin{itemize}
\item[(X)] 
$b,\sigma : [0,T]\times \real \longrightarrow \real$ are continuous in time and continuously differentiable in space for any fixed time $t$ and such that there exist $k_b,k_\sigma >0$ with 
$$|b_x(t,x)|\leq k_b,\ |\sigma_x(t,x)|\leq k_\sigma, \text{ for all $x\in\R$}.$$ 
Besides $b(t,0), \sigma(t,0)$ are bounded functions of $t$ and there exists $c>0$ such that for all $t\in [0,T]$ $$0<c\leq |\sigma(t,\cdot)|, \ \lambda(dx)-a.e.$$
\end{itemize}

\begin{Remark}\label{densite_x}According to Theorem 2.1 in \cite{FournierPrintems}, $(X)$ implies that for all $t\in(0,T]$, the law of $X_t$, denoted by $\mathcal{L}(X_t)$, has a density with respect to the Lebesgue measure.
\end{Remark}

Our results will obviously need conditions on the parameters $g$, $h$ which appear in the backward component of \eqref{edsr}. More precisely, one can distinguish between two regimes which call for two different analyses: the case where $h$ exhibits Lipschitz growth in its variables (developed in Section \ref{section:lip}), and the case where $h$ has quadratic growth in the $z$ variable (studied in Section \ref{section:quadratic}). We start with the Lipschitz situation.

\section{The Lipschitz case}
\label{section:lip}

In this section, we focus on the solution $(Y,Z)$ of FBSDE \eqref{edsr} under a Lipschitz-type assumption on the driver $h$. The problem of existence of a density for the marginal laws of $Y$ has been first studied in \cite{AntonelliKohatsu}, when the generator $h$ is assumed to be uniformly Lipschitz continuous in $y$ and $z$. We first recall in Section \ref{sub:lipprel} some general results on Lipschitz FBSDEs, then we review in Section \ref{section:lip:y} the results from \cite{AntonelliKohatsu}. Next, we point out an inefficiency in \cite[Theorem 3.6]{AntonelliKohatsu} by providing a counter example to this result, and we make precise how this small flaw can be corrected, and propose a precised version of it as Theorem \ref{AKmodifie}. Finally, in Section \ref{section:lip:z}, we study the existence of a density for the marginal laws of $Z$ when the generator $h$ of the BSDE satisfies Assumption $($L$)$. 

\subsection{Generalities on Lipschitz FBSDEs}
\label{sub:lipprel}

We start by making precise as Assumption $($L$)$ the Lipschitz condition on $h$ and the associated condition on the terminal condition $g$. We set:

\begin{itemize}
\item[(L)] 
\begin{itemize}
\item[(i)] $g : \mathbb{R} \longrightarrow \mathbb{R}$ is such that $\E[g(X_T)^2]<+\infty$.
\item[(ii)]  $h : [0,T]\times \mathbb{R}^3 \longrightarrow \mathbb{R}$ is such that there exist $(k_x,k_y,k_z)\in(\R_+^*)^3$ such that for all $(t,x_1,x_2,y_1,y_2,z_1,z_2) \in [0,T]\times \mathbb{R}^6$, 
$$ |h(t,x_1,y_1,z_1)-h(t,x_2,y_2,z_2)|\leq k_x|x_1-x_2|+k_y|y_1-y_2|+k_z|z_1-z_2|.$$
\item[(iii)] $\int_0^T |h(s,0,0,0)|^2ds<+\infty$.
\end{itemize}
\end{itemize}

Before going to the density analysis of the $Y$ and $Z$ components we recall briefly well-known facts about existence, uniqueness and Malliavin differentiability for the system \eqref{edsr} which can be found in \cite{PardouxPeng,ElkarouiPengQuenez}.

\begin{Proposition}[\cite{PardouxPeng,ElkarouiPengQuenez}]$($Existence and uniqueness$)$\label{propex}
Under Assumptions $(X)$ $($that we recall is given in Section \ref{sub:X}$)$ and $(L)$, there exists a unique solution $(X,Y,Z)$ in $\S^2 \times \S^2 \times \H^2$ to the FBSDE \eqref{edsr}.
\end{Proposition}

Concerning the Malliavin differentiability of $(X,Y,Z)$, it can obtained (see \cite{PardouxPeng} and \cite[Remark of Proposition 5.3]{ElkarouiPengQuenez}) under the following assumptions:

\begin{itemize}
\item[(D1)] 
\begin{itemize}
\item[(i)] $g$ is differentiable, $\mathcal{L}(X_T)-$a.e., $g$ and $g'$ have polynomial growth.
\item[(ii)] $(x,y,z)\mapsto h(t,x,y,z)$ is continuously differentiable for every $t$ in $[0,T]$.
\end{itemize}
\item[(D2)] 
\begin{itemize}
\item[(i)] $g$ is twice differentiable, $\mathcal{L}(X_T)-$a.e., $g$, $g'$ and $g''$ have polynomial growth. 

\item[(ii)] $(x,y,z)\mapsto h(t,x,y,z)$ is twice continuously differentiable for every $t$ in $[0,T]$.
\end{itemize}
\end{itemize}

Note that (D1) ensures that $Y$ is Malliavin differentiable, whereas $(D2)$ ensures it is twice Malliavin differentiable. As it will be made more clear below, since $Z$ can be represented as a Malliavin trace of $Y$, the fact that $Y$ is twice Malliavin differentiable entails that $Z$ is Malliavin differentiable.

\begin{Proposition}$($Malliavin differentiabiliy$)$ \label{MD}
Under $(X)$, $(L)$ and $(D1)$, we have for any $t\in[0,T]$ that $(X_t,Y_t) \in (\mathbb{D}^{1,2})^2$, $Z_t \in \mathbb{D}^{1,2}$ for almost every $t$, and for all $0<r\leq t \leq T$:
\begin{equation}\label{edsr_derive}
\begin{cases}
\displaystyle D_r X_t=\sigma(r,X_r) + \int_r^t b_x(s,X_s) D_r X_s ds + \int_r^t \sigma_x(s,X_s) D_r X_s dW_s\\
\displaystyle D_r Y_t=g'(X_T)D_rX_T+\int_t^T H(s,D_r X_s,D_r Y_s, D_r Z_s)ds -\int_t^T D_r Z_s dW_s,
\end{cases}
\end{equation} 
where $H(s,x,y,z):=h_x(s,X_s,Y_s,Z_s)x+h_y(s,X_s,Y_s,Z_s)y+h_z(s,X_s,Y_s,Z_s)z.$  
\end{Proposition}
Notice that BSDE \eqref{edsr_derive} is a linear BSDE, whose solution can be computed using the linearization method (see \cite{ElkarouiPengQuenez}).

\vspace{0.5em}
We will need extra properties on the Malliavin derivative of $Y$ and $Z$ for which the following result will be crucial. These results rely heavily on the Markovian framework we are working with.
\begin{Proposition}[\cite{MaZhang,IRR}]\label{prop:Markov}
Let Assumptions $(X)$, $(L)$ and $(D1)$ hold, then there exists a map $u:[0,T] \times \real \longrightarrow \real$ in $\Cc^{1,2}$ such that 
$$Y_t =u(t,X_t), \quad t\in [0,T], \; \P-a.s.$$
In addition, $Z$ admits a continuous version given by
\begin{equation}
\label{eq:Zu'}
Z_t = u_x(t,X_t) \sigma(t,X_t), \quad t\in [0,T], \; \P-a.s.
\end{equation}
\end{Proposition}  

In view of Proposition \ref{prop:Markov}, the chain rule formula implies that $Y_t$ belongs to $\mathbb{D}^{2,2}$ and 
\begin{equation}
\label{eq:D2Y}
D^2 Y_t = u_x(t,X_t) D^2 X_t + u_{xx}(t,X_t) (D X_t)^{\otimes 2}, \quad \P-a.s.
\end{equation}

Note that by definition, $Z$ is an element of $\H^2$. As a consequence, for any fixed element $t$ in $[0,T]$, the random variable $Z_t$ is not uniquely defined, which makes the density analysis ill-posed. However, by the previous proposition, $Z$ admits in our framework a continuous version. From now on, we will always consider this version.
\vspace{0.5em}

The following Lemma is due to Ma and Zhang in \cite[Lemma 2.4]{MaZhang} and to Pardoux and Peng \cite{PardouxPeng_92} for the representation of $Z$ as a Malliavin trace of $Y$ (see \eqref{prop:PP92} below). 

\begin{Lemma}\label{lemma_gradient_malliavin}
Let Assumptions $(X)$, $(L)$, $(D1)$ and $(D2)$ hold. Then, there exists a version of $(D_r X_t, D_r Y_t, D_r Z_t)$ for all $0<r\leq t\leq T$ which satisfies:
\begin{align*}
D_r X_t&=\nabla X_t (\nabla X_r)^{-1}\sigma(r,X_r),\ D_r Y_t&=\nabla Y_t (\nabla X_r)^{-1}\sigma(r,X_r),\ D_r Z_t&=\nabla Z_t (\nabla X_r)^{-1}\sigma(r,X_r),
\end{align*}
\begin{equation}
\label{prop:PP92}
Z_t=D_t Y_t:=\lim_{s \nearrow t} D_s Y_t,\ \mathbb{P}-a.s., \textrm{ for } a.e. \; t \in [0,T],
\end{equation}
where $(\nabla X,\nabla Y,\nabla Z)$ is the solution to the following FBSDE:
\begin{equation}\label{edsr_gradient}
\begin{cases}
\displaystyle \nabla X_t= \int_0^t b_x(s,X_s) \nabla X_s ds +\int_0^t \sigma_x(s,X_s)\nabla X_sdW_s,\\
\displaystyle \nabla Y_t = g'(X_T)\nabla X_T +\int_t^T \nabla h(s,\Theta_s)\cdot \nabla \Theta_s ds -\int_t^T \nabla Z_s dW_s.
\end{cases} 
\end{equation} 
\end{Lemma}

\begin{Remark}
Assumptions $(D1)$ and $(D2)$ are linked to the existence of first and second-order Malliavin derivatives for the $Y$ component of the solution of \reff{edsr}. We would like to point out to the reader that we only require the differentiability of $g$, $\mathcal{L}(X_T)-$a.e. Such a relaxation will be particularly useful in the quadratic case $($\textit{i.e.} in Section \ref{section:quadratic}$)$. We emphasize that when we work under Assumption $(X)$, the law of $X_T$ is absolutely continuous with respect to the Lebesgue measure and $X_T$ has finite moments of any order. Thus, thanks to standard approximation arguments, we can show that the usual chain rule formula of Malliavin calculus $($see Proposition 1.2.3. in \cite{Nualartbook}$)$ still holds for the random variable $g(X_T)$, under Assumptions $(D1)$ or $(D2)$.

%
\end{Remark}

Finally, set the following assumption
\begin{itemize}
\item[(M)] There exists a function $f\in \mathcal{C}^2(\real)$ such that for all $t\in [0,T]$: $X_t=f(t,W_t)$.
\end{itemize}
We obtain the following proposition

\begin{Proposition}\label{prop_dy_dz_r}
Under Assumptions $(M)$, $(L)$ and $(D2)$, for all $0<r,s\leq t\leq  T$ we have $D_r Y_t=D_s Y_t=Z_t$ and $D_r Z_t=D_s Z_t$, $\mathbb{P}-$a.s.
\end{Proposition}

\begin{proof}
Once again we set $\Theta_s:=(X_s,Y_s,Z_s)$. We know that for all $0<r\leq t\leq T$:
\begin{align*}
 D_r Y_t&=g'(X_T) f'(T,W_T)+\int_t^T (h_x(s,\Theta_s)f'(s,W_s)+h_y(s,\Theta_s)D_r Y_s+h_z(s,\Theta_s)D_r Z_s)ds\\
 &\hspace{1cm}- \int_t^T D_r Z_s dW_s. \end{align*}

Then $(D_r Y, D_r Z)$ satisfies a linear BSDE which does not depend on $r$ and by the uniqueness of the solution we deduce that for all $0<r,s\leq t\leq  T$ we have $D_r Y_t=D_s Y_t$ and $D_r Z_t=D_s Z_t$, $\mathbb{P}-$a.s. Finally, $D_r Y_t=Z_t$ by \eqref{prop:PP92}.
\begin{flushright}
\vspace{-1em}
$\qed$
\end{flushright}
\end{proof}

\subsection{Existence of a density for the $Y$ component}  
\label{section:lip:y}
We focus in this section on the existence of a density for the marginal laws of the process $Y$ in the Lipschitz case, pursuing the study started in \cite{AntonelliKohatsu}. Towards this goal, we recall first the so-called \textit{first order conditions} introduced in \cite{AntonelliKohatsu}, which are only sufficient, as illustrated in Example \ref{exemple}. We then turn our attention to the \textit{second-order conditions} of Theorem 3.6 in \cite{AntonelliKohatsu}. We point out a (small) inefficiency in the proof of \cite[Theorem 3.6]{AntonelliKohatsu} and provide a corrected version of this result as Theorem \ref{AKmodifie}.

\vspace{0.3em}
As in \cite{AntonelliKohatsu}, we set for any $A\in\Bc(\R)$ (i.e. the Borel $\sigma$-algebra on $\R$), and $t$ in $[0,T]$ such that $\mathbb{P}(X_T \in A | \mathcal{F}_t)>0$: 
\begin{equation}
\label{eq:barg}
\underline{g}:= \inf\limits_{x \in \real} g'(x), \quad \underline{g}^A:=\inf\limits_{x\in A} g'(x),\quad \overline{g}:= \sup\limits_{x \in \real} g'(x), \quad \overline{g}^A:=\sup\limits_{x\in A} g'(x),
\end{equation}
\begin{equation}
\label{eq:barh}
\underline{h}(t):=\inf\limits_{ s\in [t,T], (x,y,z) \in \real^3} h_x(s,x,y,z),\quad \quad \overline{h}(t):=\sup\limits_{ s\in [t,T], (x,y,z) \in \real^3} h_x(s,x,y,z).
\end{equation}

\begin{Theorem}$($First-order conditions \cite[Theorem 3.1]{AntonelliKohatsu}$)$\label{thm_H+H-}
Assume that $(X)$, $(L)$ and $(D1)$ hold. Fix some $t\in(0,T]$ and set $K:=k_b+k_y+k_{\sigma}k_z$. If there exists $A\in\mathcal B(\R)$ such that $\mathbb{P}(X_T \in A | \mathcal{F}_t)>0$ and one of the two following assumptions holds
\begin{align*}
 &(H+)\quad \begin{cases}
\displaystyle \underline{g}e^{-\sgn(\underline{g})KT}+\underline{h}(t)\int_t^T e^{-\sgn(\underline{h}(s))Ks}ds\geq0 \\
\displaystyle \underline{g}^Ae^{-\sgn(\underline{g}^A)KT}+\underline{h}(t)\int_t^T e^{-\sgn(\underline{h}(s))Ks}ds>0
\end{cases}\\[0.3em]
&(H-)\quad \begin{cases}
\displaystyle  \overline{g}e^{-\sgn(\overline{g})KT}+\overline{h}(t)\int_t^T e^{-\sgn(\overline{h}(s))Ks}ds\leq 0 \\
\displaystyle \overline{g}^Ae^{-\sgn(\overline{g}^A)KT}+\overline{h}(t)\int_t^T e^{-\sgn(\overline{h}(s))Ks}ds<0,
\end{cases}
\end{align*}

then $Y_t$ has a law absolutely continuous with respect to the Lebesgue measure.
\end{Theorem}

\begin{Remark}
Notice that $\underline{g}$ $($resp. $\overline{g})$ could be equal to $-\infty$ $($resp. $+\infty)$. Then Assumption $(H+)$ $($resp. $(H-))$ cannot be satisfied. Therefore, there is no problem if we allow the extrema of $g$ to take the values $\pm \infty$. 
\end{Remark}

\begin{Remark}\label{positivite_dy} 
In view of the proof of \cite[Theorem 3.1]{AntonelliKohatsu}, one can  show that under $(X)$, $(L)$, and $(D1)$ and if $g'\geq 0$ and $\underline{h}(t)\geq 0$ $($resp. $g'\leq 0$ and $\overline{h}(t)\leq 0)$ for $t \in [0,T]$, then for all $0<r\leq t \leq T$, $D_r Y_t \geq 0$ $($resp. $D_r Y_t \leq 0)$ and the inequality is strict if there exists $A\in\Bc(\R)$ such that $\pr(X_T \in A|\mathcal{F}_t)>0$ and $g'_{|A}>0$ $($resp. $g'_{|A}<0)$.
\end{Remark}

Note that neither Condition $(H+)$ nor Condition $(H-)$ are necessary for getting existence of a density as illustrated in the following example.

\begin{Example}\label{exemple}
Let $T=1$, $g(x)=x$, $X=W$, $h(s,x,y,z)=(s-2)x$. In this case, $K=0$ and $h_x(s,x,y,z)=s-2$ for all $(x,y,z)\in \real^3$. For any $t$ in $(0,1]$, we have:
\[ \overline{g}=\underline{g}=1,\ \underline{h}(t)=t-2,\ \overline{h}(t)=-1, \] 
so that Assumption $(H-)$ is not satisfied. Indeed,
$$ \overline{g}e^{-\sgn(\overline{g})KT}+\overline{h}(t)\int_t^T e^{-\sgn(\overline{h}(s))Ks}ds=1-(1-t)=t>0.$$

Similarly, $(H+)$  is not satisfied for any $t \in \left(0, (3-\sqrt{5})/2\right)$ since:
$$ \underline{g}e^{-\sgn(\underline{g})KT}+\underline{h}(t)\int_t^T e^{-\sgn(\underline{h}(s))Ks}ds=1+(t-2)(1-t)=-t^2+3t-1,$$ which is negative for $t \in \left(0, (3-\sqrt{5})/2\right)$. We deduce that for $t\in \left(0, (3-\sqrt{5})/2\right)$ neither Assumption $(H+)$ nor Assumption $(H-)$ is satisfied. However, we know that:
\begin{align}
\label{eq:counterex}
Y_t &= \E\left[\left.W_1+\int_t^1 (s-2)W_s ds \right| \mathcal{F}_t\right]\nonumber\\
&=W_t\left(1+\int_t^1 (s-2)ds\right)=W_t\left(-\frac12+2t-\frac{t^2}{2}\right), \ \forall t\in [0,1],\ \P-a.s.,
\end{align}
which admits a density with respect to the Lebesgue measure except when $t=0$ and $t=2-\sqrt{3}$. 
\end{Example}

Notice that in the previous example, the generator does not depend on $z$. In that setting, another result is derived \cite{AntonelliKohatsu}, involving so-called \textit{second order conditions}. There, the authors of \cite{AntonelliKohatsu} benefit from the absence of $z$ in the driver to make a higher order expansion of the Malliavin norm $\int_0^T |D_r Y_t|^2 dr$. The price to pay is that the condition involves a mapping $\tilde{h}$ (see \eqref{eq:htilde} below), which is essentially a sum of derivatives of the driver $h$, which goes beyond the simple derivative $h_x$. However, Example \ref{exemple} provides a counter-example to \cite[Theorem 3.6]{AntonelliKohatsu}. Indeed, the second-order conditions proposed in \cite[Theorem 3.6]{AntonelliKohatsu} entails that $Y_t$ admits a density, when $t\neq \frac12$, so in particular at $t=2-\sqrt{3}$. However from \eqref{eq:counterex}, $Y_{2-\sqrt{3}}=0$. This example proves that \cite[Theorem 3.6]{AntonelliKohatsu} has to be modified. The proof of \cite[Theorem 3.6]{AntonelliKohatsu} is essentially correct, except that in their proof the original Brownian motion $W$ is not a Brownian motion any more under the new measure $\mathbb Q$ defined in \cite[page 275]{AntonelliKohatsu} and need to be replaced by the process $W'_\cdot:=W_\cdot-\int_0^\cdot \sigma_x(s,X_s)ds$ which is a $\mathbb Q$-Brownian motion. This leads to the two extra terms $-(\sigma \sigma_xh_{xx}+z\sigma_xh_{xy})$ in the expression of the mapping \eqref{eq:htilde} below, compareLd to the original expression of $\tilde{h}$ in the statement of \cite[Theorem 3.6]{AntonelliKohatsu}. We refer the reader to Example \ref{counterexample} below and we propose a corrected version of \cite[Theorem 3.6]{AntonelliKohatsu} as Theorem \ref{AKmodifie} (whose proof exactly follows the original one up to the introduction of $W'$), in which the modified second-order conditions are sufficient, and necessary in the special situation of Example \ref{exemple}.

\vspace{0.3em}
Consider the FBSDE \eqref{edsr} when $h$ does not depend on $z$ and define:
\begin{align}
\label{eq:htilde}
\tilde{h}(s,x,y,z):=& -\left( h_{xt}+b h_{xx}-hh_{xy}+\frac12(\sigma^2 h_{xxx}+2z\sigma h_{xxy}+z^2h_{xxy})\right)(s,x,y)\nonumber\\
&-\left((h_y+b_x)h_x+\sigma \sigma_xh_{xx}+z\sigma_xh_{xy}\right)
(s,x,y).
\end{align}
$$ \tilde{g}(x):=g'(x)+(T-t)h_x(T,x,g(x)),$$
$$ \underline{\tilde{g}}:=\min\limits_{x\in \mathbb{R}} \tilde{g}(x), \quad \overline{\tilde{g}}:=\max\limits_{x\in \mathbb{R}} \tilde{g}(x),\quad \underline{\tilde{g}}^A:=\min\limits_{x\in A} \tilde{g}(x), \quad \overline{\tilde{g}}^A:=\max\limits_{x\in A} \tilde{g}(x),$$
$$\underline{\tilde{h}}(t):=\min\limits_{[t,T]\times \mathbb{R}^3} \tilde{h}(s,x,y,z), \ \overline{\tilde{h}}(t):=\max\limits_{[t,T]\times \mathbb{R}^3} \tilde{h}(s,x,y,z).$$
The following theorem corrects Theorem 3.6 in \cite{AntonelliKohatsu}.

\begin{Theorem}$($Second-order conditions \cite[Theorem 3.6]{AntonelliKohatsu}$)$\label{AKmodifie}
Fix some $t\in(0,T]$, assume that $h$ does not depend on $z$, that Assumptions $(X)$, $(L)$ and $(D1)$ hold and set $K:=k_y+k_b$. If there exists $A\in\mathcal B(\R)$ such that $\mathbb{P}(X_T \in A | \mathcal{F}_t)>0$ and one of the two following assumptions holds
 \begin{align*}
&\widetilde{(H+)} \quad  \begin{cases}
\displaystyle \underline{\tilde{g}}e^{-\sgn(\underline{\tilde{g}})KT}+\underline{\tilde{h}}(t)\int_t^T e^{-\sgn(\underline{\tilde{h}}(s))Ks}(T-s)ds\geq0 \\
\displaystyle \underline{\tilde{g}}^Ae^{-\sgn(\underline{\tilde{g}}^A)KT}+\underline{\tilde{h}}(t)\int_t^T e^{-\sgn(\underline{\tilde{h}}(s))Ks}(T-s)ds>0,
\end{cases}\\[0.3em]
&\widetilde{(H-)} \quad \begin{cases}
\displaystyle \overline{\tilde{g}}e^{-\sgn(\overline{\tilde{g}})KT}+\overline{\tilde{h}}(t)\int_t^T e^{-\sgn(\overline{\tilde{h}}(s))Ks}(T-s)ds\leq 0 \\
\displaystyle \overline{\tilde{g}}^Ae^{-\sgn(\overline{\tilde{g}}^A)KT}+\overline{\tilde{h}}(t)\int_t^T e^{-\sgn(\overline{\tilde{h}}(s))Ks}(T-s)ds<0,
\end{cases}
\end{align*}
then the first component $Y_t$ of the solution of BSDE \eqref{edsr} has a law which is absolutely continuous with respect to the Lebesgue measure.
\end{Theorem}

\begin{Example}\label{counterexample}
We go back to Example \ref{exemple} with $g\equiv Id.$ and $h(s,x,y,z)=(s-2)x$ which does not depend on $z$. On the one hand, we know from \eqref{eq:counterex} that for all $t\in (0,1]$, the law of $Y_t$ has a density except when $t=0$ or $t=2-\sqrt{3}$. On the other hand, our conditions in Theorem \ref{AKmodifie} read:
\[ \overline{\tilde{g}}=\tilde{\underline{g}}=\tilde{g}(x)=t, \quad \tilde{h}(t,x,y)=\overline{\tilde{h}}(t)=\underline{\tilde{h}}(t)=-1, \quad K=0,
\]
from which $\widetilde{(H+)}$ becomes:
\[ t-\int_t^1 (1-s) ds = t-(1-t)+\frac12-\frac{t^2}{2}=-\frac{t^2}{2}+2t-\frac12 >0,\]

and $\widetilde{(H-)}$ becomes:
\[ t-\int_t^1 (1-s) ds = t-(1-t)+\frac12-\frac{t^2}{2}=-\frac{t^2}{2}+2t-\frac12 <0.\]

We hence conclude, in view of Theorem \ref{AKmodifie}, that the law of $Y_t$ has a density with respect to the Lebesgue measure for every $t\in (0,1] \backslash \{2-\sqrt{3}\}$.

\vspace{0.5em}
In this particular example, notice that Theorem \ref{AKmodifie} is more accurate than Theorem \ref{thm_H+H-} since Condition $\widetilde{(H+)}$ and Condition $\widetilde{(H-)}$ are sufficient \underline{and} necessary to obtain the existence of a density for $Y$. Finally, we emphasize once more that the counterpart of Condition $\widetilde{(H-)}$ in \cite[Theorem 3.6]{AntonelliKohatsu} gives that whenever $2t-1<0$, $Y_t$ admits a density, which is clearly satisfied for $t=2-\sqrt{3}$. However we know that $Y_{2-\sqrt{3}}=0$. 
\end{Example}  

\subsection{Existence of a density for the control variable $Z$}
\label{section:lip:z}

We now turn to the problem of existence of a density for the marginal laws of $Z$. This question was studied in \cite{AbouraBourguin} when the generator is linear in $z$, that is to say $h(t,x,y,z)=\tilde{h}(t,x,y)+\alpha z$, which is from our point of view a too stringent assumption since by a Girsanov transformation this equation basically reduces to a BSDE with a generator which does not depend on $z$. We focus here on a general function $h$ satisfying Assumption (L). Consider the two following assumptions
\begin{itemize}
\item[(C+)] $h_x,h_{xx},h_{yy},h_{zz},h_{xy}\geq 0$ and $h_{xz}= h_{yz}= 0$,
\item[(C-)] $h_x,h_{xx},h_{yy},h_{zz},h_{xy}\leq 0$ and $h_{xz}= h_{yz}= 0$.
\end{itemize} 

Let $t\in (0,T]$ and $A\in\Bc(\R)$. We set:
$$ \underline{g''}:= \min\limits_{x \in \mathfrak{S}(X_T)} g''(x), \quad \underline{g''}^A:=\min\limits_{x\in \mathfrak{S}(X_T) \cap A} g''(x),\quad \underline{g'}:= \min\limits_{x \in \mathfrak{S}(X_T)} g'(x), \quad \underline{g'}^A:=\min\limits_{x\in \mathfrak{S}(X_T) \cap A} g'(x),$$
$$ \underline{h_{xx}}(t):=\min\limits_{ s\in [t,T], (x,y,z) \in \real^3} h_{xx}(s,x,y,z).$$

\begin{Theorem}\label{thm_density_z_lip}
Let Assumption $(X)$, $(L)$ and $(D2)$ hold. Let $0< t \leq T$ and assume moreover
\begin{itemize}
\item There exist $(\underline a,\overline a)\in(0,+\infty)$, such that $\underline{a}\leq  D_r X_u \leq \overline{a}$, for all $0<r<u\leq T$,
\item There exists $\overline b\geq 0$, such that $0\leq  D_{r,t}^2 X_u\leq \overline{b}$, for all $0<r,t<u\leq T$,
\item $(C+)$ holds 
\item $h_{xy}= 0$ or $(h_{xy}\geq 0$ and $g'\geq 0$, $\mathcal{L}(X_T)-a.e.)$.
\end{itemize}
If there exists a set $A\in \mathcal B(\real)$ such that $\pr(X_T\in A | \mathcal{F}_t)>0$ and such that
$$ \mathbf{1}_{\{\underline{g''}<0\}}\underline{g''}\overline{a}^2+\underline{g'}\mathbf{1}_{\{\underline{g'}<0\}}\overline{b}+(\mathbf{1}_{\{\underline{g''}\geq 0\}}\underline{g''}+\underline{h_{xx}}(t)(T-t))\underline{a}^2\geq 0,$$
and
$$ (\mathbf{1}_{\{\underline{g''}^A<0\}}\underline{g''}^A\overline{a}^2+\underline{g'}^A\mathbf{1}_{\{\underline{g'}<0\}}\overline{b}) +(\mathbf{1}_{\{\underline{g''}^A\geq 0\}}\underline{g''}^A+\underline{h_{xx}}(t)(T-t))\underline{a}^2>0, $$

then, the law of $Z_t$ has a density with respect to the Lebesgue measure.
\end{Theorem}

\begin{proof} Under the assumptions of Theorem \ref{thm_density_z_lip}, we obtain for $0<r,s< t\leq T$:
\begin{align*}
D_{r,s}^2 Y_t=&\ g''(X_T)D_r X_T D_s X_T+g'(X_T)D_{r,s}^2 X_T -\int_t^T D_{r,s}^2 Z_u dW_u\\
&+\int_t^T[h_x(u,\Theta_u)D_{r,s}^2 X_u+h_{xx}(u,\Theta_u)D_r X_u D_s X_u+h_{xy}(u,\Theta_u) D_s X_u D_r Y_u\\
&+h_y(u,\Theta_u) D_{r,s}^2 Y_u +h_{xy}(u,\Theta_u) D_r X_u D_s Y_u+h_{yy}(u,\Theta_u) D_r Y_u D_s Y_u\\
&+h_z(u,\Theta_u) D_{r,s}^2 Z_u+h_{zz}(u,\Theta_u)D_r Z_u D_s Z_u] du.
\end{align*}  
Let $\tilde{\pr}$ be the probability equivalent to $\pr$ such that
\begin{equation} \label{tilde_p}
\frac{d\tilde{\pr}}{d\pr}=\exp\left(\int_0^T h_z(s,\Theta_s) dW_s-\frac12\int_0^T\abs{h_z(s,\Theta_s)}^2ds\right),
\end{equation} 
where $h_z$ is bounded thanks to Assumption (L).
Under $\tilde{\pr}$ defined by \eqref{tilde_p}, we obtain:
\begin{align*}D_{r,s}^2 Y_t&=\mathbb{E}^{\tilde{\pr}}\Big[g''(X_T)D_r X_T D_s X_T+g'(X_T)D_{r,s}^2 X_T\\
& \hspace{1cm}+\int_t^T[h_x(u,\Theta_u)D_{r,s}^2 X_u+h_{xx}(u,\Theta_u)D_r X_u D_s X_u+h_{xy}(u,\Theta_u) D_s X_u D_rY_u\\
&\hspace{2cm}+h_y(u,\Theta_u) D_{r,s}^2 Y_u +h_{xy}(u,\Theta_u) D_r X_u D_s Y_u+h_{yy}(u,\Theta_u) D_r Y_u D_s Y_u\\
&\hspace{2cm}+h_{zz}(u,\Theta_u)D_r Z_u D_s Z_u] du\Big|\Fc_t\Big]. 
\end{align*}
By standard linearization techniques, we obtain:
\begin{align*}D_{r,s}^2 Y_t&=\mathbb{E}^{\tilde{\pr}}\Big[e^{\int_t^T h_y(u,\Theta_u)du}(g''(X_T)D_r X_T D_s X_T+g'(X_T)D_{r,s}^2 X_T) \\
&\hspace{1cm}+\int_t^T e^{\int_t^u h_y(v,\Theta_v)dv}[h_x(u,\Theta_u)D_{r,s}^2 X_u+h_{xx}(u,\Theta_u)D_r X_u D_s X_u\\
&\hspace{1cm}\hspace{1cm} + h_{xy}(u,\Theta_y) (D_r X_u D_s Y_u +D_s X_u D_r Y_u)\\
&\hspace{1cm}\hspace{1cm}+h_{yy}(u,\Theta_u) D_r Y_u D_s Y_u+h_{zz}(u,\Theta_u)D_r Z_u D_s Z_u] du\Big|\mathcal F_t\Big]. \end{align*}

Then, using Remark \ref{positivite_dy}, Lemma \ref{lemma_gradient_malliavin} and our assumptions we obtain: 
\begin{align*}
&e^{\int_t^T h_y(u,\Theta_u)du}(g''(X_T)D_r X_T D_s X_T+g'(X_T)D_{r,s}^2 X_T)\\
&\quad +\int_t^T e^{\int_t^u h_y(v,\Theta_v)dv}[h_x(u,\Theta_u)D_{r,s}^2 X_u+h_{xx}(u,\Theta_u)D_r X_u D_s X_u\\
&\quad+ h_{xy}(u,\Theta_y) (D_r X_u D_s Y_u +D_s X_u D_r Y_u)\\
&\quad +h_{yy}(u,\Theta_u) D_r Y_u D_s Y_u+h_{zz}(u,\Theta_u)D_r Z_u D_s Z_u] du\\
&\geq e^{\int_t^T h_y(u,\Theta_u)du} \left(\mathbf{1}_{\{\underline{g''}<0\}}\underline{g''}\overline{a}^2+\underline{g'}\mathbf{1}_{\{\underline{g'}<0\}}\overline{b}+(\mathbf{1}_{\{\underline{g''}\geq 0\}}\underline{g''}+\underline{h_{xx}}(t)(T-t))\underline{a}^2\right)\geq 0.
\end{align*}

We deduce that:
\begin{align*}D_{r,s}^2 Y_t\geq&\ \mathbb{E}^{\tilde{\pr}}\Big[e^{\int_t^T h_y(u,\Theta_u)du}\mathbf{1}_{X_T\in A}(g''(X_T)D_r X_T D_s X_T+g'(X_T)D_{r,s}^2 X_T) \\
&+\mathbf{1}_{X_T\in A}\int_t^T e^{-K(u-t)}[h_{xx}(u,\Theta_u)\underline{a}^2] du\Big|\mathcal{F}_t\Big]\\
\geq &\ e^{-KT}\left(\mathbf{1}_{\{\underline{g''}^A<0\}}\underline{g''}^A\overline{a}^2+\underline{g'}^A\mathbf{1}_{\{\underline{g'}<0\}}\overline{b}\right)\mathbb{\tilde{\pr}}(X_T\in A | \mathcal{F}_t)\\
& +e^{-KT}\left(\mathbf{1}_{\{\underline{g''}^A\geq 0\}}\underline{g''}^A+\underline{h_{xx}}(t)(T-t)\right)\underline{a}^2\mathbb{\tilde{\pr}}(X_T\in A | \mathcal{F}_t).
\end{align*}

Using the fact that $D^2 Y_t$ is symmetric, the chain rule formula, \eqref{eq:Zu'} and \eqref{eq:D2Y} and the fact that $\lim_{s \nearrow t} D_{r,s}^2 X_t = \sigma'(t,X_t) D_r X_t$, we have that $ \lim_{s \nearrow t} D^2_{r,s} Y_t =D_r Z_t,$ from which we deduce that $D_r Z_t>0$, $\P-a.s.$ Then according to Bouleau and Hirsch's Theorem, we conclude that the law of $Z_t$ has a density with respect to the Lebesgue measure.
\begin{flushright}
\vspace{-1em}
$\qed$
\end{flushright}
\end{proof}

\begin{Remark}\label{ab}
Notice that the sign assumption on $D^2X$ can be obtained under the following sufficient conditions.
\begin{itemize}
\item[$(X+)$] For any $t\in[0,T]$, the maps $x\longmapsto b(t,x)$ and $x\longmapsto \sigma(t,x)$ are respectively in $\mathcal C^2(\R)$ and $\mathcal C^3(\R)$, and there exists $c>0$ such that $$\sigma\geq c>0,\quad \sigma' \geq 0,\quad \sigma'',\sigma'''\leq 0 \; \text{ and } \; [\sigma,[\sigma,b]]\geq 0,$$ where $[b,\sigma]$ denotes the Lie bracket between $b$ and $\sigma$ defined by $[b,\sigma]:=b'\sigma+\sigma' b$.
\item[$(X-)$] For any $t\in[0,T]$, the maps $x\longmapsto b(t,x)$ and $x\longmapsto \sigma(t,x)$ are respectively in $\mathcal C^2(\R)$ and $\mathcal C^3(\R)$, and there exists $c<0$ such that 
$$\sigma\leq c<0,\ \sigma' \leq 0,\ \sigma'',\sigma'''\geq 0\text{ and }[\sigma,[\sigma,b]]\leq 0.$$
\end{itemize}

Indeed, according to the first step of the proof of Theorem 4.3 in \cite{AbouraBourguin}, Condition $(X+)$ $($resp. $(X-))$ ensures that $D^2 X$ is non-negative $($resp. non-positive$)$.\end{Remark}

\begin{Remark}
\label{rem:X+-}
One can provide an alternative version of the previous result, whose proof follows the same lines as the one of Theorem \ref{thm_density_z_lip}. Fix $t$ in $(0,T]$, let Assumptions $(L)$, $(X)$ and $(D2)$ hold and assume that there exists $A\in\Bc(\real)$ such that $\mathbb{P}(\left.X_T\in A  \right|  \mathcal{F}_t)>0$, and such that one of the two following conditions is satisfied:
\begin{itemize}
\item[$a)$] $(X+)$ and $(C+)$ hold true and $g''\geq0, \; g''_{\vert A}>0 \text{ and } g' \geq0, \quad \mathcal{L}(X_T)\text{-a.e.}$
\item[$b)$] $(X-)$ and $(C-)$ hold true and $g''\leq0, \; g''_{\vert A}<0 \text{ and } g' \leq0,  \quad \mathcal{L}(X_T)\text{-a.e.},$
\end{itemize}
then, for all $t \in (0,T]$, the law of $Z_t$ has a density with the respect to Lebesgue measure.
\end{Remark}

When Assumption (M) holds, Theorem \ref{thm_density_z_lip} takes a different form as shown below in Theorem \ref{densite_z_f}, mainly because of Proposition \ref{prop_dy_dz_r}. Indeed, consider the following assumptions:
\begin{itemize}
\item[($\tilde{C}+$)] $h_{zz}\geq 0$ and $h_{xz}=h_{yz}\equiv 0$.
\item[($\tilde{C}-)$] $h_{zz}\leq 0$ and $h_{xz}=h_{yz}\equiv 0$.
\end{itemize}

Under Assumption ($\tilde{C}+)$ or ($\tilde{C}-)$, we recall that:
\begin{align*}\label{drz}
D_r Z_t=&\ g''(X_T)|f'(T,W_T)|^2+g'(X_t)f''(T,W_T)\\
&+\int_t^T \Big[h_x(u,\Theta_u)f''(u,W_u)+h_y(u,\Theta_u) D_{r,t}^2 Y_u \Big]du\\
&+\int_t^T\Big[|f'(u,W_u)|^2 h_{xx}(u,\Theta_u)+\left(h_{xy}(u,\Theta_u)D_r Y_u+ D_t Y_uh_{xy}(u,\Theta_u) \right) f'(u,W_u)\Big]du\\
&+\int_t^T\Big[h_{yy}(u,\Theta_u) \underbrace{D_t Y_uD_r Y_u}_{=|Z_u|^2}+\underbrace{D_t Z_u D_r Z_u}_{=|D_r Z_u|^2} h_{zz}(u,\Theta_u) \Big]du - \int_t^T D_{r,t}^2 Z_u d\tilde{W}_u,
\end{align*}
with $\tilde{W}:=W-\int_0^\cdot h_z(s,\Theta_s) ds$.
We set $\theta=(x,y,z),$ and
\begin{align*}
\tilde{h}(t,w,x,y,z,\tilde{z}):= &\ h_{xx}(t,\theta)|f'(t,w)|^2+h_x(t,\theta)f''(t,w)+(h_{yy}(t,\theta)z+2h_{xy}(t,\theta)f'(t,w))z\\
&+h_y(t,\theta)\tilde{z},
\end{align*}
$$ \underline{\tilde{h}}(t)=\min\limits_{(s,w,x,y,z,\tilde{z})\in [t,T]\times \real^5} \tilde{h}(s,w,x,y,z,\tilde{z}), \quad \overline{\tilde{h}}(t)=\max\limits_{(s,w,x,y,z,\tilde{z})\in [t,T]\times \real^5} \tilde{h}(s,w,x,y,z,\tilde{z}).$$

\vspace{0.3em}
\begin{Theorem}\label{densite_z_f}
Assume that $(M)$, $(L)$ and $(D2)$ are satisfied and that there exists $A\in\Bc(\R)$ such that $\mathbb{P}(X_T\in A | \mathcal{F}_t)>0$ and one of the two following assumptions holds:

\begin{itemize}
\item[$a)$] Assumption $(\tilde{C}+)$, $ \underline{((g'\circ f) f')'}+(T-t)\underline{\tilde{h}}(t) \geq 0$ and 
$\underline{((g'\circ f) f')'_A} +(T-t)\underline{\tilde{h}}(t)>0.$

\item[$b)$] Assumption $(\tilde{C}-)$, $\overline{((g'\circ f) f')'} +(T-t)\overline{\tilde{h}}(t)\leq 0$ and $\overline{((g'\circ f) f')'_A}+(T-t)\overline{\tilde{h}}(t) <0.$
\end{itemize}
Then, the law of $Z_t$ is absolutely continuous with respect to the Lebesgue measure on $\real$.
\end{Theorem}

\begin{proof}
Using Proposition \ref{prop_dy_dz_r}, we recall that:
\begin{align*}
D_r Z_t=&\ g''(X_T)|f'(T,W_T)|^2+g'(X_t)f''(T,W_T)\\
&+\int_t^T \tilde{h}(u,W_u,X_u,Y_u,Z_u,D_r Z_u)+|D_r Z_u|^2 h_{zz}(u) du - \int_t^T D_{r,t}^2 Z_u d\tilde{W}_u,
\end{align*}
where $\tilde{W}:=W-\int_0^\cdot h_z(u,\Theta_u) du$.
Then the proof follows exactly the same line as the one of Theorem \ref{thm_density_z_lip}.
\begin{flushright}
\vspace{-1em}
$\qed$
\end{flushright}
\end{proof}


\section{The quadratic case}
\label{section:quadratic}

We now turn to the quadratic case and provide an extension of both Theorem \ref{thm_H+H-} and Theorem \ref{thm_density_z_lip}. Note however that the assumptions of these theorems do not find immediate counterparts in the quadratic setup since the latter involves the Lipschitz constant of $h$ with respect to the $z$ variable (see Remark \ref{rk:condi}). We also emphasize that existence of densities for the $Y$ and $Z$ components in the quadratic case that we consider here was open until now. We first make precise the quadratic growth setting together with existence, uniqueness and Malliavin differentiability results for these equations in the next section. Then, we investigate respectively in Sections \ref{section:quadratic:y} and \ref{section:quadratic:z} the existence of density for respectively $Y$ and $Z$.

\subsection{Generalities on quadratic FBSDEs}
\label{sub:quadprel}

In contadistinction to the previous section, we will now assume that $h$ exhibits quadratic growth in the $z$ variable. As noted in the introduction, this case is particularly useful for applications, especially in Finance where any pricing and hedging problem on an incomplete market which can be translated into a BSDE analysis will lead to a quadratic BSDE. The precise assumption for dealing with quadratic BSDEs is given as:

\begin{itemize}
\item[(Q)]
\begin{itemize}
\item[(i)] $g : \mathbb{R} \longrightarrow \mathbb{R}$ is bounded.
\item[(ii)] $h : [0,T]\times \mathbb{R}^3 \longrightarrow \mathbb{R}$ is such that:
\begin{itemize}
\item[$\triangleright$] There exists $(K,K_z,K_y)\in(\R_+^*)^3$ such that for all $(t,x,y,z) \in [0,T]\times \mathbb{R}^3$
$$\hspace{-3em}|h(t,x,y,z)|\leq K(1+|y|+|z|^2),\ \abs{h_z}(t,x,y,z)\leq K_z(1+|z|),\ \abs{h_y}(t,x,y,z)\leq K_y.$$
\item[$\triangleright$]  There exists $C>0$ such that for all $(t,x,y,z_1,z_2) \in [0,T] \times \real^4$
\begin{equation}\label{condition_h} |h(t,x,y,z_1)-h(t,x,y,z_2)| \leq C(1+|z_1|+|z_2|) |z_1-z_2|.\end{equation}
\end{itemize}
\item[(iii)] $\int_0^T |h(s,0,0,0)|^2ds<+\infty$.
\end{itemize}
\end{itemize}

Existence and uniqueness of a solution triplet $(X,Y,Z)$ under Assumption $(Q)$ has been obtained in \cite{Kobylanski}. More precisely:

\begin{Proposition}[\cite{Kobylanski}]$($Existence and uniqueness of BSDEs$)$\label{propexq}
Under Assumptions $(X)$ and $(Q)$, there exists a unique solution $(X,Y,Z)$ in $\S^2 \times \mathbb S^\infty\times\mathbb H^2_{\rm BMO}$.  
\end{Proposition}
Note that Condition \eqref{condition_h} on the generator $h$ in Assumption $(Q)$ in the one that ensures uniqueness of the solution. Hence, it can be dropped and one can then consider the maximal solution $Y$ of the BSDE, for which our proofs still apply.

\vspace{0.3em}
Concerning the Malliavin differentiability of the processes $(X,Y,Z)$ it has been obtained in the quadratic case in \cite{AIDR} under the Assumptions $(D1)$ and $(D2)$ (that are defined in Section \ref{sub:lipprel}). Note that Proposition \ref{prop:Markov} still holds true if Assumption $(L)$ is replaced by Assumption $(Q)$. However, although the above proposition is completely proved in \cite{MaZhang} in the Lipschitz case, we did not find a proper reference in the quadratic case, except for \cite{IRR} which proves the result under Assumption $(Q)$, with the exception that $u$ is only shown to be in $\Cc^{1,1}$. Nonetheless, one can still obtain the required result by proving that Theorem $3.1$ of \cite{MaZhang} still holds for a BSDE with a driver which is uniformly Lipschitz in $y$ and stochastic Lipschitz in $z$ with a Lipschitz process in $\mathbb H^2_{\rm BMO}$ $($which is exactly the case of the BSDE satisfied by the Malliavin derivative of $Y)$. This can be achieved by following exactly the steps of the proof of Theorem 3.1 in \cite{MaZhang}, where the a priori estimates of their Lemma $2.2$ have to be replaced by those given in Lemma $A.1$ of \cite{IRR}.  As in the Lipschitz case, Relation \eqref{eq:D2Y} still holds true under $(Q)$. In addition, as for Proposition \ref{prop:Markov}, the proof of Lemma \ref{lemma_gradient_malliavin} can be extended to the quadratic setting. Finally, Propositions \ref{MD} and \ref{prop_dy_dz_r} are valid if one replaces Assumption $(L)$ by Assumption $(Q)$.

\begin{Proposition}$($Malliavin differentiabiliy$)$ \label{MDq}
Under $(X)$, $(Q)$ and $(D1)$, we have for any $t\in[0,T]$ that $(X_t,Y_t) \in (\mathbb{D}^{1,2})^2$, $Z_t \in \mathbb{D}^{1,2}$ for almost every $t$, and for all $0<r\leq t \leq T$:
\begin{equation}\label{edsr_derive1}
\begin{cases}
\displaystyle D_r X_t=\sigma(r,X_r) + \int_r^t b_x(s,X_s) D_r X_s ds + \int_r^t \sigma_x(s,X_s) D_r X_s dW_s\\
\displaystyle D_r Y_t=g'(X_T)D_rX_T+\int_t^T H(s,D_r X_s,D_r Y_s, D_r Z_s)ds -\int_t^T D_r Z_s dW_s,
\end{cases}
\end{equation} 
where $H(s,x,y,z):=h_x(s,X_s,Y_s,Z_s)x+h_y(s,X_s,Y_s,Z_s)y+h_z(s,X_s,Y_s,Z_s)z.$  
\end{Proposition}

\subsection{Existence of a density for the $Y$ component} 
\label{section:quadratic:y}
\begin{Theorem}\label{thm_H+H-_quadra}
Fix $t\in (0,T]$ and assume that $(X)$, $(Q)$ and $(D1)$ hold. If there is $A\in\mathcal B(\R)$ such that $\mathbb{P}(X_T\in A \ | \ \mathcal{F}_t)>0$ and one of the following assumptions holds $($see Definitions \eqref{eq:barg}-\eqref{eq:barh}$)$
\begin{itemize}
\item[$(Q+)$] $g'\geq 0$ and  $g'_{\vert A} >0$, $\mathcal{L}(X_T)-$a.e. and $\underline{h}(t)\geq 0$,
\item[$(Q-)$] $g'\leq 0,\ g'_{\vert A}<0$, $\mathcal{L}(X_T)-$a.e. and $\overline{h}(t)\leq 0$,
\end{itemize}
then $Y_t$ has a law absolutely continuous with respect to the Lebesgue measure.
\end{Theorem}

\begin{proof} 
To simplify the notations for any $s$ in $[0,T]$, we set $\Theta_s:=(X_s,Y_s,Z_s)$. We set $K:=k_b\vee k_y\vee k_\sigma$. We assume that (Q+) is satisfied (the proof with (Q-) follows the same lines, so we omit it). According to Bouleau-Hirsch's criterion, it is enough to show that $\gamma_{Y_t}:=\int_0^T |D_r Y_t|^2 dr >0$, $\mathbb{P}$-a.s. As in the proof of \cite[Theorem 3.6]{AntonelliKohatsu}, we have for $0\leq r \leq t \leq T$, that $D_r Y_t$ writes down as:
\begin{equation}\label{eq:y}
D_r Y_t=g'(X_T)D_r X_T+\int_t^T 
h_x(s,\Theta_s)D_r X_s+h_y(s,\Theta_s)D_r Y_s ds +\int_t^T D_r Z_s dW_s.
\end{equation}
From \reff{eq:y}, and following the expression of $\gamma_{Y_t}$ given in \cite[page 271]{AntonelliKohatsu}, we deduce that  
$$ \gamma_{Y_t}=\left(\mathbb{E}\left[ g'(X_T)\zeta_T \psi_T +\int_t^T \psi_s h_x(s,\Theta_s)\zeta_s ds | \mathcal{F}_t \right]\right)^2 (\psi_t^{-1})^2 \int_0^t (\zeta_r^{-1}\sigma(r,X_r))^2 dr,$$
with 
$$ \psi_t\zeta_t=\underbrace{e^{\int_0^t (b_x(s,X_s)+h_y(s,\Theta_s)+\sigma_x(s,X_s)h_z(s,\Theta_s))ds}}_{=:E_t} \underbrace{e^{\int_0^t (\sigma_x(s,X_s)+h_z(s,\Theta_s)) dW_s - \frac12 \int_0^t (\sigma_x(s,X_s)+h_z(s,\Theta_s))^2 ds}}_{=:M_t}.$$  
Let $\mathbb{Q}$ the probability measure equivalent to $\pr$ with density $\frac{d\mathbb{Q}}{d\pr}:=M_T$. Indeed, $M$ is a martingale as $\int_0^\cdot (\sigma_x(s,X_s)+h_z(s,\Theta_s)) dW_s$ is a BMO martingale due to the boundedness of $\sigma_x$ (by  (X)) and the fact that $|h_z(s,\Theta_s)|\leq C(1+|Z_s|)$ (by (Q)) and from the BMO property of $\int_0^\cdot Z_s dW_s$ (by Proposition \ref{propex}). We therefore have:
$$ \mathbb{E}\left[ g'(X_T)\psi_T\zeta_T+\int_t^T \psi_s h_x(s,\Theta_s)\zeta_s ds \Big|\mathcal{F}_t\right]=M_t\mathbb{E}^{\mathbb{Q}}\left[ g'(X_T) E_T +\int_t^T h_x(s,\Theta_s) E_s ds \Big| \mathcal{F}_t\right].$$

Using (Q+), we know that:
$$g'(X_T)E_T+\int_t^T h_x(s,\Theta_s) E_s ds \geq \underline{g}E_T+\underline{h}(t)\int_t^T E_sds\geq 0. $$
Thus,
\begin{align*}
&\mathbb{E}\left[ g'(X_T)\psi_T\zeta_T+\int_t^T \psi_s h_x(s,\Theta_s)\zeta_s ds \Big|\mathcal{F}_t\right]\\
&\geq M_t \mathbb{E}^{\mathbb{Q}}\left[ \mathbf{1}_{X_T \in A}\left(g'(X_T) E_T +\int_t^T h_x(s,\Theta_s) E_s ds\right) \Big| \mathcal{F}_t\right]\\
&\geq  M_t\Big(\underline{g}^Ae^{-2KT}\mathbb{E}^\mathbb{Q}\left[ \mathbf{1}_{X_T\in A} e^{-K\int_0^T |h_z(s,\Theta_s)|ds}\big| \mathcal{F}_t\right]\\
&\hspace{0.9em}+\underline{h}(t)e^{-2KT}(T-t)\mathbb{E}^\mathbb{Q}\left[ \mathbf{1}_{X_T\in A} e^{-K\int_0^T|h_z(s,\Theta_s)|ds} \big| \mathcal{F}_t \right]\Big)\\
&\geq M_t\Big(\underline{g}^Ae^{-2KT}\mathbb{E}^\mathbb{Q}\Big[ \mathbf{1}_{X_T\in A} e^{-K\sqrt{T}\sqrt{\int_0^T |h_z(s,\Theta_s)|^2ds}}\big| \mathcal{F}_t\Big]\\
&\hspace{0.9em}+\underline{h}(t)e^{-2KT}(T-t)\mathbb{E}^\mathbb{Q}\Big[ \mathbf{1}_{X_T\in A} e^{-K\sqrt{T}\sqrt{\int_0^T|h_z(s,\Theta_s)|^2ds} }\big| \mathcal{F}_t \Big]\Big),
\end{align*}

where the last inequality is due to Cauchy-Schwarz inequality. Besides, according to Assumption (Q), $|h_z(s,\Theta_s)|\leq C(1+|Z_s|).$ Then, we deduce that $\int_0^T |h_z(s,\Theta_s)|^2ds<+\infty, \ \mathbb{P}-$a.s., since $Z\in \mathbb{H}^2$. Hence, $M_t>0$, $\mathbb{P}-$a.s.
Given that the law of $X_T$ is absolutely continuous with respect to the Lebesgue measure, we deduce that $\mathbb{E}\left[ g'(X_T)\psi_T\zeta_T+\int_t^T \psi_s h_x(s,\Theta_s)\zeta_s ds \Big|\mathcal{F}_t\right]>0, \ \mathbb{P}-\text{a.s.}$
We conclude using Theorem \ref{BH}.
\begin{flushright}
\vspace{-1em}
$\qed$
\end{flushright}
\end{proof}

\begin{Remark}
Similarly to Remark \ref{positivite_dy}, the proof of Theorem \ref{thm_H+H-_quadra} shows that under $(X)$, $(Q)$, $(D1)$ and if $g'\geq 0$ and $\underline{h}(t)\geq 0$ $($resp. $g'\leq 0$ and $\overline{h}(t)\leq 0)$ for $t \in [0,T]$, then for all $0<r\leq t \leq T$, $D_r Y_t \geq 0$ $($resp. $D_r Y_t \leq 0)$ and the inequality is strict if there exists $A\in\Bc(\R)$ such that $\pr(X_T \in A|\mathcal{F}_t)>0$ and $g'_{|A}>0$ $($resp. $g'_{|A}<0)$.
 
\end{Remark}
\begin{Remark}\label{rk:condi}
Conditions $(Q+)$ and $(Q-)$ are stronger than $(H+)$ and $(H-)$, due to the unboundedness of $h_z$, which prevents us from reproducing the same proof than in \cite{AntonelliKohatsu}. Indeed, in this framework the quantity appearing for instance in $(H+)$ becomes:
$$ \underline{g}e^{-2K\sgn(\underline{g})T}e^{-K\sgn(\underline{g})\int_0^T |h_z(s)|ds}+\underline{h}(t)e^{-2K\sgn(\underline{h}(t))T}\int_t^T e^{-K\sgn(\underline{h}(t))\int_0^s |h_z(s)|ds},$$
whose sign for every $K\geq 0$ depends strongly on those of $g'$ and $h_x$. This is why we must use the stronger conditions $(Q+)$ and $(Q-)$.
\end{Remark}

\begin{Remark}
In \cite[Corollary 3.5]{dosdos} comonotonicty conditions on the data of a BSDE under Assumption $(Q)$ are given so that $Z_t \geq 0$, $\P-$a.s., $\forall t \in [0,T]$. In addition, the authors claim that strict comonotonicity entails that $Z_t>0$, which implies by Bouleau-Hirsch criterion that the law of $Y_t$ has a density with respect to the Lebesgue measure. However, we do not understand their proof and it is not true that an increasing mapping which is differentiable has a positive derivative everywhere $($even if one relaxes it by asking for a positive derivative Lebesgue-almost everywhere$)$ and one needs an extra assumption to prove that the derivative does not vanish. Indeed, take any closed set of positive Lebesgue measure with empty interior $($for instance the Smith-Volterra-Cantor set on $\R)$. By Whitney's extension Theorem, there exists a differentiable increasing map whose derivative vanishes on this set.  
\end{Remark}

\subsection{Existence of a density for the control variable $Z$}
\label{section:quadratic:z}

In this section, we obtain existence results for the density of $Z$ under Assumption (Q). We actually have exactly the same type of results as in the Lipschitz case with similar proofs, which highlights the robustness and flexibility of our approach. Let us detail first the changes that we have to make. 

\vspace{0.3em}

Under (Q), using the fact that for all $s\in [0,T]$ $|h_z(s,\Theta_s)|\leq C(1+|Z_s|)$ and according to Proposition \ref{propex} we deduce that $\int_0^\cdot h_z(s,\Theta_s) dW_s$ is a BMO-martingale. Then, according to Theorem 2.3 in \cite{Kazamaki}, the stochastic exponential of $\int_0^\cdot h_z(s,\Theta_s) dW_s$ is  a uniformly integrable martingale and we can apply Girsanov's Theorem. We also emphasize that in (Q), $g$ is not assumed to be twice continuously differentiable. Indeed, to recover the BMO properties linked to quadratic BSDEs (and thus in order to be able to apply the above reasoning), $g$ needs to be bounded, which is incompatible with g convex (or concave). Nevertheless, there exist terminal conditions $g$ which are twice differentiable almost everywhere on the support of the law of $X_T$ (which is some closed subset of $\R$), such that their second-order derivative have a given sign there. As an example, take $X=W$ and $g(x):= f(x)\mathbf{1}_{x\in[a,b]}+f(a)\mathbf{1}_{x\leq a} + f(b) \mathbf{1}_{x\leq b}$ with $f$ a twice differentiable convex function and $a,b \in \mathbb{R}$.

\begin{Theorem}\label{thm_density_z_quadra}
Let Assumptions $(X)$, $(Q)$ and $(D2)$ hold. Let $0< t \leq T$ and assume moreover
\begin{itemize}
\item There exist $(\underline a,\overline a)$ s.t., $0<\underline{a}\leq  D_r X_u \leq \overline{a}$, for all $0<r<u\leq T$.
\item There exists $\overline b$ s.t., $0\leq  D_{r,s}^2 X_u\leq \overline{b}$, for all $0<r,s<u\leq T$.
\item $(C+)$ holds and $h_y\geq 0$.
\item $h_{xy}= 0$ or $(h_{xy}\geq 0$ and $g'\geq 0$, $\mathcal{L}(X_T)$-a.e.$)$.
\end{itemize}
If there exists $A\in\Bc(\R)$ such that $\pr(X_T\in A | \mathcal{F}_t)>0$ and such that:
$$ \mathbf{1}_{\{\underline{g''}<0\}}\underline{g''}\overline{a}^2+\underline{g'}\mathbf{1}_{\{\underline{g'}<0\}}\overline{b}+(\mathbf{1}_{\{\underline{g''}\geq 0\}}\underline{g''}+\underline{h_{xx}}(t)(T-t))\underline{a}^2\geq 0,$$

and
$$ (\mathbf{1}_{\{\underline{g''}^A<0\}}\underline{g''}^A\overline{a}^2+\underline{g'}^A\mathbf{1}_{\{\underline{g'}<0\}}\overline{b}) +(\mathbf{1}_{\{\underline{g''}^A\geq 0\}}\underline{g''}^A+\underline{h_{xx}}(t)(T-t))\underline{a}^2>0,$$

then, the law of $Z_t$ has a density with respect to the Lebesgue measure.
\end{Theorem}

\begin{proof} As in the proof of Theorem \ref{thm_density_z_lip}, we notice that for all $0<r,t\leq s\leq T$:
\begin{align*}D_{r,s}^2 Y_t&=\mathbb{E}^{\tilde{\pr}}\Big[g''(X_T)D_r X_T D_s X_T+g'(X_T)D_{r,s}^2 X_T\\
& \hspace{1cm}+\int_t^T[h_x(u,\Theta_u)D_{r,s}^2 X_u+h_{xx}(u,\Theta_u)D_r X_u D_s X_u+h_y(u,\Theta_u)D_r Y_u D_s Y_u \\
&\hspace{1cm}\hspace{1cm}+h_{yy}(u,\Theta_u)D^2_{r,s}Y_u+h_{zz}(u,\Theta_u)D_r Z_u D_s Z_u] du\Big|\Fc_t\Big], \end{align*}
where $\tilde{\pr}$ is the equivalent probability measure to $\pr$ with density $$\frac{d\tilde{\pr}}{d\pr}:=\exp\left(\int_0^T h_z(u,\Theta_u)dW_u-\frac12\int_0^T \abs{h_z(u,\Theta_u)}^2du\right),$$ given that $\int_0^\cdot h_z(u,\Theta_u)dW_u$ is a BMO-martingale and using Theorem 2.3 in \cite{Kazamaki}. Then the proof is similar to that of Theorem \ref{thm_density_z_lip}.
\begin{flushright}
\vspace{-1em}
$\qed$
\end{flushright}
\end{proof}

\begin{Remark} In order to satisfy the condition in Theorem \ref{thm_density_z_quadra}, there are basically two types of sufficient conditions
\begin{itemize}
\item First of all, if the support of the law of $X_T$ is bounded from above, then one can take $g$ to continuously differentiable everywhere, non-decreasing, convex and bounded on this support. Then it suffices to take $h$ to be convex in $x$ as well.

\item However, when the support of the law of $X_T$ is no longer bounded from above, then it is no longer possible to find $g$ which is non-decreasing, bounded and convex on this support. We must therefore allow $g''$ to become non-positive, and the role of $h_{xx}$ becomes then crucial, as it has to be sufficiently positive in order to balance $g''$. As an example, take $X:=W$. Then $\overline a=\underline a =1$ and $\overline b=0$. One can choose $g(x):= \frac{1}{1+x^2}$. Then, there exists a positive constant $M$ such that $-2\leq g''(x)\leq M$ and by choosing $h$ such that $h$ satisfies the assumptions in Theorem \ref{thm_density_z_quadra} and $t\in(0,T)$ such that $\underline{h_{xx}}(t)(T-t)\geq 2$, we deduce that $Z_t$ admits a density.
\end{itemize}
\end{Remark}
We give also a theorem under Assumption (M):

\begin{Theorem}\label{densite_z_f_quadra}
Assume that $(M)$, $(Q)$ and $(D2)$ are satisfied and that there exists $A\in\Bc(\R)$ such that $\mathbb{P}(X_T\in A | \mathcal{F}_t)>0$ and one of the two following assumptions holds:
\begin{itemize}
\item[$a)$] Assumption $(\tilde{C}+)$, $\underline{((g'\circ f) f')'}+(T-t)\underline{\tilde{h}}(t) \geq 0$ and 
$ \underline{((g'\circ f) f')'_A} +(T-t)\underline{\tilde{h}}(t)>0.$
\item[$b)$] Assumption $(\tilde{C}-)$, $\overline{((g'\circ f) f')'} +(T-t)\overline{\tilde{h}}(t)\leq 0$ and $\overline{((g'\circ f) f')'_A}+(T-t)\overline{\tilde{h}}(t) <0.$
\end{itemize}

Then, the law of $Z_t$ is absolutely continuous with respect to the Lebesgue measure.
\end{Theorem}

The proof is the same as the proof of Theorem \ref{densite_z_f} using the BMO property of $\int_0^\cdot Z_s dW_s$, we therefore omit it. We now turn to the simplest case of quadratic growth BSDE and verify that it is covered by our result.
\begin{Example}
Let us consider the following BSDE
\[ Y_t=g(W_T)+\int_t^T \frac12 |Z_s|^2ds -\int_t^T Z_s dW_s,\]
where $g$ is bounded. According to Theorem \ref{thm_density_z_quadra} with $\overline{a}=\underline{a}=1$, $\overline{b}=0$ and $h_{xx}=0$, we deduce that for all $t\in (0,T]$, the law of $Z_t$ has a density with respect to the Lebesgue measure if $g''\geq 0$, $\lambda(dx)$-a.e. and if there exists $A\in\Bc(\R)$ with positive Lebesgue measure such that $g''_{|A}>0$.

\vspace{0.3em}
We emphasize that, as a sanity check, this can be verified by direct calculations. Indeed, using the fact that if $F \in \mathbb{D}^{1,2}$ then $D_r (\E[F|\mathcal{F}_t])=\E[D_r F|\mathcal{F}_t] \mathbf{1}_{[0,t)} (r)$ $($see \cite[Proposition 1.2.4]{Nualartbook}$)$ we deduce that if $0\leq r<t\leq T$ then:
$$D_r Y_t = \frac{\E[g'(W_T)e^{g(W_T)}| \mathcal{F}_t]}{\E[e^{g(W_T)}| \mathcal{F}_t]},$$which does not depend on $r$. Then according to Proposition \ref{prop:PP92},
$Z_t=  \frac{\E[g'(W_T)e^{g(W_T)}| \mathcal{F}_t]}{\E[e^{g(W_T)}| \mathcal{F}_t]}.$

Take $0<r<t\leq T$, then:
$$D_r Z_t=  \frac{\E[g''(W_T)e^{g(W_T)}+ |g'(W_T)|^2 e^{g(W_T)} | \mathcal{F}_t]\E[e^{g(W_T)}| \mathcal{F}_t]-|\E[g'(W_T)e^{g(W_T)}| \mathcal{F}_t]|^2}{\E[e^{g(W_T)}| \mathcal{F}_t]}.$$
Using Cauchy-Schwarz inequality, if $g''\geq 0$, $\lambda(dx)$-a.e. and if there exists $A\in\Bc(\R)$ with positive Lebesgue measure such that $g''_{|A}>0$, we deduce that for all $t\in (0,T]$, $Z_t$ has a density with respect to the Lebesgue measure by Theorem \ref{BH}.
\end{Example}

\section{Density estimates for the marginal laws of $Y$ and $Z$}
\label{section:densY}
Up to now, the density estimates obtained in the literature relied mainly on the fact that the framework considered implied that the Malliavin derivative of $Y$ was bounded. Hence, using the Nourdin-Viens' formula (or more precisely their Corollary 3.5 in \cite{NourdinViens}), it could be showed that the law of $Y$ has Gaussian tails. Although such an approach is perfectly legitimate from the theoretical point of view, let us start by explaining why, as pointed out in the introduction, we think that this is not the natural framework to work with when dealing with BSDEs. Consider indeed the following example.

\begin{Example}\label{rem.toostringent}
Let us consider the FBSDE \reff{edsr}, with $T=1$, $g(x):=x^3$, $h(t,x,y,z):=3x$, $b(t,x)=0$, $\sigma(t,x)=1$ and $X_0=0$. Then, simple computations show that the unique solution is given by
$$X_t=W_t,\quad Y_t=W_t^3+6W_t(1-t),\quad Z_t=3W_t^2+6(1-t).$$
Then, both $Y_t$ and $Z_t$ have a law which is absolutely continuous with respect to the Lebesgue measure, for every $t\in(0,1]$, but neither $Y_t$ nor $Z_t$ has Gaussian tails.
\end{Example}

Moreover, when it comes to applications dealing with generators with quadratic growth, assuming that the Malliavin derivative of $Y$ is bounded implies that the process $Z$ itself is bounded as $Z_t=D_t Y_t$, which is seldom satisfied in applications, since in general, one only knows that $Z\in\H^2_{\rm BMO}$. 

\vspace{0.3em}
One of the main applications of the results we obtain in this section is the precise analysis of the error in the truncation method in numerical schemes for quadratic BSDEs, introduced in \cite{ImkellerDosreis} and studied in \cite{ChassagneuxRichou}. We recall that according to Proposition \ref{prop:Markov} there exists a function $v : [0,T]\times \real\longmapsto \real$ in $\Cc^{1,2}$ such that $Y_t=v(t,X_t)$ and $Z_t=v_x(t,X_t)\sigma(t,X_t)$. Since we want to study the tails of the laws of $Y$ and $Z$, we will assume from now on that the support of these laws is $\real$, which implies that neither $v$ nor $v'$ is bounded from below or above. Moreover, we emphasize that throughout this section, we will assume that $Y_t$ and $Z_t$ do have a law which is absolutely continuous, so as to highlight the conditions needed to obtain the estimates. Throughout this section we assume that $X_t=W_t$ in \eqref{edsr} (that is $X_0=0, \; \sigma\equiv 1, \; b \equiv 0$).

\subsection{Preliminary results}
We will have to study the asymptotic growth of $v$ and $v_x$ in the neighborhood of  $\pm \infty$. To this end, we introduce for any measurable function $f:\R\longrightarrow \R$ the following two kinds of growth rates:
\[  \overline{\alpha_f}:=\inf\left\{\alpha>0,\ \limsup\limits_{|x| \to +\infty} \abs{\frac{f(x)}{x^\alpha}}<+\infty\right\}, \quad \underline{\alpha_f}:=\inf\left\{\alpha>0,\ \liminf\limits_{|x| \to +\infty} \abs{\frac{f(x)}{x^\alpha}}<+\infty\right\}.  \]

\begin{Lemma}\label{lemme_v-1}
Let $ f \in \mathcal{C}^1(\real)$. Assume that for all $x \in \real$, $f'(x)>0$. If $0<\underline{\alpha_f}<+\infty$ then for all positive constant $0<\eta<\underline{\alpha_f}$: 
$$\overline{\alpha_{f^{(-1)}}}\leq \frac{1}{\underline{\alpha_f}-\eta},$$

where $f^{(-1)}$ is the inverse function of $f$.
\end{Lemma}

\begin{proof}
Using the definition of $\underline{\alpha_f}$, we deduce that for all $\eta>0$,
\begin{equation*}
\liminf\limits_{|x| \to + \infty} \abs{\dfrac{f(x)}{x^{\underline{\alpha_f}-\eta}}}=\lim\limits_{|x| \to + \infty} \abs{\dfrac{f(x)}{x^{\underline{\alpha_f}-\eta}}}=+\infty. 
\end{equation*}

Since $f$ and $f^{(-1)}$ are increasing and unbounded from above and below, we deduce that there exists $\overline{x}>0$ such that for all $x\geq \overline{x}$, $f(x)$ and $f^{(-1)}$ are positive. Then, for all $M>0$, there exists $x_0 \geq \overline{x}$ such that for all $x\geq x_0>0$ and for all $y \geq M x_0^{\underline{\alpha_f}-\eta}\vee \overline x$
\begin{align*}
 f(x)\geq M x^{\underline{\alpha_f}-\eta}\Longleftrightarrow &\; f\left((y M^{-1})^\frac{1}{\underline{\alpha_v}-\eta}\right)\geq y\Longleftrightarrow(y M^{-1})^{\frac{1}{\underline{\alpha_f}-\eta}} \geq f^{(-1)}(y).
\end{align*}
This implies directly that $\limsup\limits_{y \to +\infty} \abs{\frac{f^{(-1)}(y)}{y^{\frac{1}{\underline{\alpha_f}-\eta}}}} <+\infty.$ The proof is similar when $y$ goes to $ -\infty$.
\begin{flushright}
\vspace{-1em}
$\qed$
\end{flushright}
\end{proof}

It is rather natural to expect that for well-behaved functions $f \in \mathcal{C}^1(\real)$,  $\overline{\alpha_f}=\underline{\alpha_f}$ and $\overline{\alpha_f}=\overline{\alpha_{f'}}+1$. However, the situation is unfortunately not that clear. First of all, this may not be true if $f$ is not monotone. Indeed, let $f(x):=x^2\sin(x)$, then $\overline{\alpha_f}=\underline{\alpha_f}=2$. Furthermore, the strict monotonicity of $f$ is not sufficient either. Without being completely rigorous, let us describe a counterexample. Consider a function $f$ defined on $\R_+$, equal to the identity on $[0,1]$, which then increases as $x^4$ until it crosses $x\longmapsto x^2$ for the first time, which then increases as $x^{1/2}$ until it crosses $x\longmapsto x$ for the first time and so on. Finally, extend it by symmetry to $\R_-$. Then, it can be checked that $\overline{\alpha_f}=2$, $\underline{\alpha_f}=1$, $\overline{\alpha_{f'}}=3$, $\underline{\alpha_{f'}}=0$.

\vspace{0.3em}
A nice sufficient condition for the aforementioned result to hold is that $f'$ is a \textit{regularly varying function} (see \cite{BinghamGoldieTeugels} and \cite{Seneta}).
\begin{Lemma}\label{prop_regularly}
Assume that $f'$ is equivalent in $+ \infty$ $($resp. in $-\infty)$ to a regularly varying function with Karamata's decomposition $x^\beta L_1(x)$ where $L_1$ is slowly varying $($resp. $x^\beta L_2(x)$ where $L_2$ is slowly varying$)$ and where $\beta>0$. Then
\begin{itemize}
\item[$\rm{(i)}$] $f$ is equivalent in $+ \infty$ $($resp. in $-\infty)$ to a regularly varying function with Karamata's decomposition $x^{\beta+1} \widetilde{L_1}(x)$ where $\widetilde{L_1}$ is slowly varying $($resp. $x^{\beta+1} \widetilde{L_2}(x)$ where $\widetilde{L_2}$ is slowly varying$)$.
\item[$\rm{(ii)}$] $  \overline{\alpha_f}=\underline{\alpha_f}=\underline{\alpha_{f'}}+1=\overline{\alpha_{f'}}+1.$
\end{itemize}
\end{Lemma}

\begin{proof} By Karamata's Theorem (see Theorem 1.5.11 in \cite{BinghamGoldieTeugels} with $\sigma=1$), for any $x_0 \in \real$:
\begin{equation}
\label{eq:cvreg}
\frac{xf'(x)}{f(x)-f(x_0)} \longrightarrow \beta +1, \ \text{ when } x\longrightarrow +\infty.
\end{equation}
In addition, $f'$ is equivalent to a regularly varying function with Karamata's decomposition $x^\beta L_1(x)$ when $x\longrightarrow +\infty$, hence in view of \eqref{eq:cvreg}, there exists a function $\widetilde{L_1}$ (equivalent to a constant times $L_1$ at $+\infty$) slowly varying such that $f$ is equivalent when $x\longrightarrow +\infty$ to a regularly varying function with Karamata's decomposition $x^{\beta+1}\widetilde{L_1}(x)$. The same result holds when $x\longrightarrow -\infty$. 

\vspace{0.5em}
We now show (ii). According to Proposition 1.3.6 (v) in \cite{BinghamGoldieTeugels} and (i), we deduce that:
$$ \overline{\alpha_{f}}=\beta+1=\underline{\alpha_{f}} \ \text{ and } \ \overline{\alpha_{f'}}=\beta=\underline{\alpha_{f'}} .$$
\begin{flushright}
\vspace{-1em}
$\qed$
\end{flushright}
\end{proof}

\subsection{A general estimate}

From now on, for a map $(t,x)\longmapsto v(t,x)$, $v'(t,x)$ will denote for simplicity the derivative of $v$ with respect to the space variable. Before enonciating a general theorem which gives us density estimates for the tails of the law of random variables of the form $v(t,W_t)$ and will be used to obtain estimates for the laws of $Y_t$ and $Z_t$, we set some 
constants in order to simplify the notations in Theorem \ref{thm_estime_y} below.

\paragraph*{List of constants} Let $\alpha\in (0,+\infty)$, $\alpha' \in\R_+$ and $\tilde\alpha>0$. For $\varepsilon >0$, we set
$$C_{\varepsilon,v,\alpha}:=\sup\limits_{x\in \R,\; t\in [0,T]} \frac{\abs{v(t,x)}}{1+|x|^{\alpha+\varepsilon}},\ \delta_{\alpha'}:= \max(1,2^{\alpha'}), \ \Xi_{\alpha'}:=\frac{\alpha'\Gamma\left(\frac{1+\alpha'}{2}\right)}{2\sqrt{\pi}},\ \mu(\tilde\alpha):=\int_\R\frac{\phi(z)}{1+|z|^{\tilde{\alpha}}} dz,$$
$$ D_{\alpha'}:=\max\left(1+\delta_{\alpha'}\Xi_{\alpha'} +\frac{\delta_{\alpha'}^2}{2}\left( \Xi_{\alpha'} + (1+\alpha')^{-1}\right)^2, \frac12 + \frac{\delta_{\alpha'}}{1+\alpha'}\right),$$
where $\Gamma$ is the usual Euler function and $\phi$ the distribution function of the normal law, defined by
$$\Gamma(x):=\int_0^{+\infty} e^{-t}t^{x-1}dt,\ x> 0, \ \text{and }\phi(x):=\frac{1}{\sqrt{2\pi}}e^{-\frac{x^2}{2}},\ x\in\mathbb R.$$

We emphasize that the following theorem can be applied in much more general cases, and it is clearly not limited to the context of BSDEs. It could for instance be used to provide non-Gaussian tail estimates for the law of solutions to some SDEs. Therefore, it has an interest of its own.
\begin{Theorem}\label{thm_estime_y}
Fix $t\in (0, T]$. Let $v: [0,T]\times \real \longrightarrow \real$ in $\mathcal{C}^{1,1}$ and let $P_t:=v(t,W_t)$. Assume furthermore that $P_t\in L^1(\P)$, that $v$ is unbounded in $x$ both from above and from below, that $v'>0$, $\underline{\alpha_v}\in (0,+\infty)$, $\overline{\alpha_{v'}}<+\infty$ and that there exist $\tilde{\alpha}> 0$ and $K>0$ such that: 
\begin{equation}\label{ineg_v'}
\frac{1}{v'(t,x)}\leq K(1+|x|^{\tilde{\alpha}}),\text{ for all $x\in\R$}.
\end{equation}
Then, the law of $P_t$ has a density with respect to the Lebesgue measure, denoted by $\rho_t$, and for all $\varepsilon, \varepsilon' >0$ and for every $y\in\R$
\begin{equation}\label{estime_y_alpha}
\rho_t(y)\leq \frac{\E[|P_t-\E[P_t]|]}{2M(\varepsilon')t}\left(1+|y|^{2\tilde{\alpha}{(\overline{\alpha_{v^{(-1)}}}+\varepsilon')}}\right)\exp \left(-\int_0^{y-\E[P_t]} \frac{(M'(\varepsilon,\varepsilon')t)^{-1}xdx}{1+|x+\E[P_t]|^{2(\overline{\alpha_{v'}}+\varepsilon)(\overline{\alpha_{v^{(-1)}}}+\varepsilon')}}\right),
\end{equation} and
\begin{equation}\label{estime_y_alpha_bis}
\rho_t(y)\geq \frac{(2M'(\varepsilon,\varepsilon')t)^{-1}\E[|P_t-\E[P_t]|]}{1+|y|^{2(\overline{\alpha_{v'}}+\varepsilon)(\overline{\alpha_{v^{(-1)}}}+\varepsilon')}}\exp \left(-\int_0^{y-\E[P_t]} \frac{x\left(1+|x+\E[P_t]|^{2\tilde{\alpha}(\overline{\alpha_{v^{(-1)}}+\varepsilon')}}\right)dx}{M(\varepsilon')t}\right),
\end{equation}
with $$M'(\varepsilon, \varepsilon'):=C_{\varepsilon,v',\overline{\alpha_{v'}}}^2 D_{\overline{\alpha_{v'}}+\varepsilon}\left(1+C_{\varepsilon',v^{(-1)}, \overline{\alpha_{v^{(-1)}}}}^{2(\overline{\alpha_{v'}}+\varepsilon)}\right) \delta_{2(\overline{\alpha_{v'}}+\varepsilon)} ,$$
 and $$ M(\varepsilon'):= \frac{\mu(\tilde\alpha)}{K^2 \left(1+C^{2\tilde \alpha}_{\varepsilon', v^{(-1)}, \overline{\alpha_{v^{(-1)}}}}\delta_{2\tilde\alpha}\right)} ,$$ using the aforementioned definitions of the constants.
\end{Theorem}

\vspace{0.3em}
\textbf{Proof.} Notice immediately that since the map $x\longmapsto v(t,x)$ is in $\mathcal C^1(\R)$ and increasing, the law of $P_t$ clearly has a density. We prove inequalities \eqref{estime_y_alpha} and \eqref{estime_y_alpha_bis} using Nourdin and Viens' formula (see Theorem \ref{thm_NourdinViens}).The rest of the proof is divided into three steps.

\vspace{0.3em}
{\bf Step $1$:} Given that for all $0<r\leq t \leq T$, $D_r P_t=v'(t,W_t)$, the function $g_{P_t}$ defined by \eqref{gzt} becomes
$$ g_{P_t}(y):= \int_0^\infty e^{-a} \mathbb{E}\left[\mathbb{E}^*[\langle \Phi_{P_t}(W),\widetilde{\Phi_{P_t}^a}(W)\rangle_{\h}] | P_t-\mathbb{E}[P_t]=y\right] da, \quad y\in \real, $$
 with\footnote{Knowing that $D_r P_t$ does not depend on $r$, $\Phi_{P_t}(W):[0,T]\longrightarrow L^2(\Omega,\Fc,\P)$ is a random process which is actually constant on $[0,t]$.} $\Phi_{P_t}(W):=v'(t,W_t)$ and where $\widetilde{\Phi_{P_t}^a}(W):=\Phi_{P_t}(e^{-a}W+\sqrt{1-e^{-2a}}W^*)$ with $W^*$ an independent copy of $W$ defined on a probability space $(\Omega^*,\mathcal{F}^*,\mathbb{P}^*)$ where $\mathbb{E}^*$ is the expectation under $\mathbb{P}^*$ ($\Phi_{P_t}$ being extended on $\Omega\times \Omega^*$). Letting $\phi(z):=\frac{1}{\sqrt{2\pi t}} e^{-\frac{z^2}{2t}}$, we get that
\begin{align} \label{expression_gy}
\nonumber g_{P_t}(y)&= \int_0^\infty e^{-a} \mathbb{E}\left[\mathbb{E}^*[\langle \Phi_{P_t}(W),\widetilde{\Phi_{P_t}^a}(W)\rangle_{\h}] | W_t=v^{(-1)}(t,y+\mathbb{E}[P_t])\right] da, \quad y\in \real,\\
&=tv'(t,v^{(-1)}(t,y+\mathbb{E}[P_t]))\int_0^\infty e^{-a} \int_\real v'\Big(t,e^{-a} v^{(-1)}(t,y+\E[P_t])+\sqrt{1-e^{-2a}}z\Big)\phi(z) dz da.
\end{align} 

\vspace{0.3em}
{\bf Step $2$: Upper bound for $g_{P_t}$}

\vspace{0.3em}
Recall that for all $\varepsilon >0$:
 $$0<v'(t,x)\leq C_{\varepsilon,v',\overline{\alpha_{v'}}}\left(1+|x|^{\overline{\alpha_{v'}}+\varepsilon}\right), \quad \forall x \in \real.$$
Then, using \eqref{expression_gy} we get:
\begin{align*}
g_{P_t}(y)\leq &\ C_{\varepsilon,v',\overline{\alpha_{v'}}}^2 t  \left(1+\abs{v^{(-1)}\left(y+\mathbb{E}[P_t]\right)}^{\overline{\alpha_{v'}}+\varepsilon}\right)\\
&\times \int_0^{+\infty} e^{-a}\int_\real\left(1+ \abs{e^{-a} v^{(-1)}(t,y+\E[P_t])+\sqrt{1-e^{-2a}}z}^{\overline{\alpha_{v'}}+\varepsilon}\right)\phi(z) dz da\\
\leq &\ C_{\varepsilon,v',\overline{\alpha_{v'}}}^2 t \left(1+|v^{(-1)}(y+\mathbb{E}[P_t])|^{(\overline{\alpha_{v'}}+\varepsilon)}\right)\\
&\times \int_0^{+\infty} e^{-a}\int_\real\left(1+\delta_{\overline{\alpha_{v'}}+\varepsilon}\left(e^{-a(\overline{\alpha_{v'}}+\varepsilon)} \abs{ v^{(-1)}(t,y+\E[P_t])}^{\overline{\alpha_{v'}}+\varepsilon}+\abs{z}^{\overline{\alpha_{v'}}+\varepsilon}\right)\right)\phi(z) dz da\\
\leq & \ C_{\varepsilon,v',\overline{\alpha_{v'}}}^2 t \left(1+|v^{(-1)}(y+\mathbb{E}[P_t])|^{\overline{\alpha_{v'}}+\varepsilon}\right)\\
&\times \left(1+ \frac{\delta_{\overline{\alpha_{v'}}+\varepsilon}}{1+\overline{\alpha_{v'}}+\varepsilon} |v^{(-1)}(y+\mathbb{E}[P_t])|^{\overline{\alpha_{v'}}+\varepsilon} +\delta_{\overline{\alpha_{v'}}+\varepsilon} \Xi _{\overline{\alpha_{v'}}+\varepsilon}\right)\\
\leq &  \ C_{\varepsilon,v',\overline{\alpha_{v'}}}^2 t D_{\overline{\alpha_{v'}} +\varepsilon}\left( 1+ |v^{(-1)}(y+\mathbb{E}[P_t])|^{2(\overline{\alpha_{v'}}+\varepsilon)}\right).
\end{align*}

By Lemma \ref{lemme_v-1}, $\overline{\alpha_{v^{(-1)}}}$ belongs to $(0,+\infty)$, hence by the definition of $\overline{\alpha_{v^{(-1)}}}$ it holds for all $\varepsilon'>0$ that
\begin{equation}\label{upper_bound_gy}g_{P_t}(y)\leq M'(\varepsilon,\varepsilon')t\left(1+|y+\E[P_t]|^{2(\overline{\alpha_{v'}}+\varepsilon)(\overline{\alpha_{(v)^{-1}}}+\varepsilon')}\right).
\end{equation}

{\bf Step $3$: Lower bound for $g_{P_t}$}

\vspace{0.3em}
Using Assumption \eqref{ineg_v'} and \eqref{expression_gy} we have that
\begin{align*}
g_{P_t}(y)\geq &\ \frac{t}{K^2(1+|v^{(-1)}(t,y+\E[P_t])|^{\tilde{\alpha}})}\\
&\times\int_0^{+\infty} e^{-a}\int_\real \frac{1}{1+|e^{-a} (v)^{-1}(t,y+\E[P_t])|^{\tilde{\alpha}}+|\sqrt{1-e^{-2a}}z|^{\tilde{\alpha}}}\phi(z) dz da.
\end{align*}
Noticing that $|\sqrt{1-e^{-2a}}z|^{\tilde{\alpha}}\leq |z|^{\tilde{\alpha}}$, and that 
$$\int_\real \frac{(1+|x|^{\tilde{\alpha}})\phi(z)}{1+|x|^{\tilde{\alpha}}+|z|^{\tilde{\alpha}}} dz\geq \mu(\tilde\alpha), \quad \forall x\in \real$$
we deduce that:
\begin{align*}
g_{P_t}(y)&\geq \frac{\mu(\tilde\alpha)t}{K^2(1+|v^{(-1)}(t,y+\E[P_t])|^{\tilde{\alpha}})}\int_0^{+\infty} e^{-a}\frac{1}{1+e^{-a\tilde{\alpha}}|v^{(-1)}(t,y+\E[P_t])|^{\tilde{\alpha}}}da.
\end{align*}
Hence:
$
g_{P_t}(y)\geq \frac{\mu(\tilde\alpha)t}{K^2(1+|v^{(-1)}(t,y+\E[P_t])|^{2\tilde{\alpha}})}.
$
We finally get Relation \eqref{estime_y_alpha_bis} for 
$$M(\varepsilon'):= \frac{\mu(\tilde\alpha)}{K^2 \left(1+C^{2\tilde \alpha}_{\varepsilon', v^{(-1)}, \overline{\alpha_{v^{(-1)}}}}\delta_{2\tilde\alpha}\right)}.$$
We conclude using Nourdin and Viens' formula.
\begin{flushright}
\vspace{-1em}
$\qed$
\end{flushright}

\begin{Corollary}\label{cor_estime_y}
Let the assumptions in Theorem \ref{thm_estime_y} hold, with the same notations. Assume moreover that $0\leq \overline{\alpha_{v'}}<\underline{\alpha_v}<+\infty $. Then there exist $\varepsilon_0,\varepsilon'_0>0$, $y_0>0$ and $\gamma\in(0,1)$ such that for any $|y|> y_0$:
\begin{equation}\label{estime_y_alpha_sansint}
\rho_t(y)\leq  \frac{\E[|P_t-\E[P_t]|]}{2M(\varepsilon'_0)t}\left(1+|y|^{2\tilde{\alpha}{(\overline{\alpha_{v^{(-1)}}}+\varepsilon'_0)}}\right)\exp\left(-\frac{\abs{y-\E[P_t]}^{2(1-\gamma)}-\abs{y_0-\E[P_t]}^{2(1-\gamma)}}{4(1-\gamma)tM'(\varepsilon_0,\varepsilon'_0)}\right),
\end{equation} 
and
\begin{align}\label{estime_y_alpha_bis_sansint}
\nonumber\rho_t(y)\geq &\ \frac{\E[|P_t-\E[P_t]|]}{2M'(\varepsilon_0,\varepsilon'_0)t\left(1+|y|^{\gamma}\right)}\exp\left(-\frac{\abs{y-\E[P_t]}^{2(\tilde\alpha(\overline{\alpha_{v^{(-1)}}}+\varepsilon'_0)+1)}-\abs{y_0-\E[P_t]}^{2(\tilde\alpha(\overline{\alpha_{v^{(-1)}}}+\varepsilon'_0)+1)}}{M(\varepsilon_0')t(\tilde\alpha(\overline{\alpha_{v^{(-1)}}}+\varepsilon'_0)+1)}\right)\\
&\times \exp\left(-\frac{\abs{y_0-\E[P_t]}^2}{M(\varepsilon_0')t}\left(1+y_0^{2\tilde\alpha\left(\overline{\alpha_{v^{(-1)}}}+\varepsilon'_0\right)}\right)\right).
\end{align}
\end{Corollary}

\begin{proof}
Let us define for any $\varepsilon,\varepsilon'>0$ 
$$\gamma(\varepsilon,\varepsilon'):=(\overline{\alpha_{v'}}+\varepsilon) (\overline{\alpha_{v^{(-1)}}}+\varepsilon').$$ 
Since we assumed that $0\leq \overline{\alpha_{v'}}<\underline{\alpha_v}<+\infty $, we can deduce using Lemma \ref{lemme_v-1} that there exist some $\varepsilon_0,\varepsilon_0'>0$ such that
$$\gamma:=\gamma(\varepsilon_0,\varepsilon_0')<1.$$

We start with \reff{estime_y_alpha_sansint}. We have from Theorem \ref{thm_estime_y}
$$ \rho_t(y)\leq \frac{\E[|P_t-\E[P_t]|]}{2M(\varepsilon'_0)t}\left(1+|y|^{2\tilde{\alpha}{(\overline{\alpha_{v^{(-1)}}}+\varepsilon'_0)}}\right)\exp \left(-\int_0^{y-\E[P_t]} \frac{xdx}{M'(\varepsilon_0,\varepsilon'_0)t\left(1+|x+\E[P_t]|^{2\gamma}\right)}\right).$$

We notice that $$\lim\limits_{\abs{x}\to +\infty}\frac{x}{M'(\varepsilon_0,\varepsilon_0')t(1+|x+\E[P_t]|^{2\gamma})} \times \frac{1}{\frac{x}{M'(\varepsilon_0,\varepsilon_0')t|x|^{2\gamma}}}=1,$$ so that there exists $x_0$ large enough such that $\frac{x}{M'(\varepsilon_0,\varepsilon_0')t(1+|x+\E[P_t]|^{2\gamma})} \geq \frac{x}{2M'(\varepsilon_0,\varepsilon_0')t|x|^{2\gamma}}$ when $\abs{x}\geq x_0$. Hence, since $\gamma\in(0,1)$, we know that we can find some $y_0>0$ large enough such that if $\abs{y}>y_0$ 
\begin{align*}
& \int_{y_0-\E[P_t]}^{y-\E[P_t]} \frac{xdx}{M'(\varepsilon_0,\varepsilon_0')t(1+|x+\E[P_t]|^{2\gamma})}\\&\geq \int_{y_0-\E[P_t]}^{y-\E[P_t]} \frac{xdx}{2M'(\varepsilon_0,\varepsilon_0')t|x|^{2\gamma}}\\
 &=\frac{1}{4(1-\gamma)tM'(\varepsilon_0,\varepsilon'_0)}\left(\abs{y-\E[P_t]}^{2(1-\gamma)}-\abs{y_0-\E[P_t]}^{2(1-\gamma)}\right),
 \end{align*}

from which \reff{estime_y_alpha_sansint} follows directly. Similarly, increasing $y_0$ if necessary, we have that for $\abs{y}>y_0$
\begin{align*}
&\int_0^{y-\E[P_t]}x\left(1+|x+\E[P_t]|^{2\tilde{\alpha}(\overline{\alpha_{v^{(-1)}}}+\varepsilon'_0)}\right)dx\\
&=\underbrace{\int_0^{y_0-\E[P_t]} x\left(1+|x+\E[P_t]|^{2\tilde{\alpha}(\overline{\alpha_{v^{(-1)}}}+\varepsilon'_0)}\right)dx}_{:=I_1}+\underbrace{\int_{y_0-\E[P_t]}^{y-\E[P_t]} x\left(1+|x+\E[P_t]|^{2\tilde{\alpha}(\overline{\alpha_{v^{(-1)}}}+\varepsilon'_0)}\right)dx}_{:=I_2}.\\
\end{align*}

\vspace{-2.5em}
Using the fact that the function $x\longmapsto 1+|x+\E[P_t]|^{2\tilde{\alpha}(\overline{\alpha_{v^{(-1)}}}+\varepsilon'_0)} $ is convex, we deduce that for $y_0$ large enough
$$I_1\leq \abs{y_0-\E[P_t]}^2\left(1+y_0^{2\tilde\alpha\left(\overline{\alpha_{v^{(-1)}}}+\varepsilon'_0\right)}\right).   $$
Moreover, since $\lim\limits_{x\to +\infty} x\left(1+|x+\E[P_t]|^{2\tilde{\alpha}(\overline{\alpha_{v^{(-1)}}+\varepsilon'_0)} +\varepsilon'_0)}\right) \times \frac{1}{x^{2{\tilde{\alpha}(\overline{\alpha_{v^{(-1)}}}+\varepsilon'_0)+1}}}=1 $, we obtain for $x$ large enough
$$ x\left(1+|x+\E[P_t]|^{2\tilde{\alpha}(\overline{\alpha_{v^{(-1)}}+\varepsilon'_0)} +\varepsilon'_0)}\right)\leq 2x^{2{\tilde{\alpha}(\overline{\alpha_{v^{(-1)}}}+\varepsilon'_0)+1}}. $$
Then, we have that for $|y|\geq y_0$
$$ I_2\leq \frac{\abs{y-\E[P_t]}^{2(\tilde\alpha(\overline{\alpha_{v^{(-1)}}}+\varepsilon'_0)+1)}-\abs{y_0-\E[P_t]}^{2(\tilde\alpha(\overline{\alpha_{v^{(-1)}}}+\varepsilon'_0)+1)}}{\tilde\alpha(\overline{\alpha_{v^{(-1)}}}+\varepsilon'_0)+1}.$$
Hence,
\begin{align*}
&\int_0^{y-\E[P_t]}x\left(1+|x+\E[P_t]|^{2\tilde{\alpha}(\overline{\alpha_{v^{(-1)}}}+\varepsilon'_0)}\right)dx\\
&\leq \abs{y_0-\E[P_t]}^2\left(1+y_0^{2\tilde\alpha\left(\overline{\alpha_{v^{(-1)}}}+\varepsilon'_0\right)}\right) +\frac{\abs{y-\E[P_t]}^{2(\tilde\alpha(\overline{\alpha_{v^{(-1)}}}+\varepsilon'_0)+1)}-\abs{y_0-\E[P_t]}^{2(\tilde\alpha(\overline{\alpha_{v^{(-1)}}}+\varepsilon'_0)+1)}}{\tilde\alpha(\overline{\alpha_{v^{(-1)}}}+\varepsilon'_0)+1},
\end{align*}

 from which the second inequality \reff{estime_y_alpha_bis_sansint} follows directly using \reff{estime_y_alpha_bis}.
\begin{flushright}
\vspace{-1em}
$\qed$
\end{flushright}
 \end{proof}

Finally, we have the following theorem, which is a simple application of the results obtained above in the special cases where we take the random variables $(Y_t,Z_t)$ solutions to the BSDE \reff{edsr} when they can be written $Y_t=v(t,W_t)$ and $Z_t=v'(t,W_t)$.

\begin{Theorem}\label{estim.dens}
Let $(Y,Z)$ be the solution to the BSDE \reff{edsr} $($which is assumed to exist and to be unique$)$. Assume  that there exists a map $v\in\mathcal C^{1,2}$ such that $Y_t=v(t,W_t)$. 
\begin{itemize}
\item[$\rm{(i)}$] If in addition, $v'>0$, $0\leq \overline{\alpha_{v'}}<\underline{\alpha_v}<+\infty $ and there exist $K>0$, $\tilde\alpha>0$ such that $v'(t,x)\geq 1/(K(1+\abs{x}^{\tilde\alpha}))$ then, denoting $\rho_{Y_t}$ the density of the law of $Y_t$, there exist $y_0>0$, $C_1,C_2>0$, $p_1\in(0,2)$ and $p_2>0$ $($which are given explicitly in Theorem \ref{thm_estime_y}$)$ such that for any $\abs{y}>y_0$
\begin{align*}
\rho_{Y_t}(y)\geq&\ \frac{\E[\abs{Y_t-\E[Y_t]}}{C_2t\left(1+\abs{y}^{1-p_1/2}\right)}\exp\left(-\frac{\abs{y-\E[Y_t]}^{2(p_2+1)}-\abs{y_0-\E[Y_t]}^{2(p_2+1)}}{(p_2+1)C_2t}\right)\\
\rho_{Y_t}(y)\leq&\ \frac{\E[\abs{Y_t-\E[Y_t]}}{C_1t}\left(1+\abs{y}^{2p_2}\right)\exp\left(-\frac{2\abs{y_0-\E[Y_t]}^2}{C_2t}\left(1+y_0^{2p_2}\right)\right)\\
&\times\exp\left(-\frac{\abs{y-\E[Y_t]}^{p_1}-\abs{y_0-\E[Y_t]}^{p_1}}{p_1C_2t}\right).
\end{align*}

\item[$\rm{(ii)}$] If in addition, $v''>0$, $0\leq \overline{\alpha_{v''}}<\underline{\alpha_{v'}}<+\infty $ and there exist $K>0$, $\tilde\alpha>0$ such that $v''(t,x)\geq 1/(K(1+\abs{x}^{\tilde\alpha}))$ then, denoting $\rho_{Z_t}$ the density of the law of $Z_t$, there exists $Z_0>0$, $C_1,C_2>0$, $p_1\in(0,2)$ and $p_2>0$ $($which are given explicitly in Theorem \ref{thm_estime_y}$)$ such that for any $\abs{z}>z_0$
\begin{align*}
\rho_{Z_t}(z)\geq&\ \frac{\E[\abs{Z_t-\E[Z_t]}}{C_2t\left(1+\abs{z}^{1-p_1/2}\right)}\exp\left(-\frac{\abs{z-\E[Z_t]}^{2(p_2+1)}-\abs{z_0-\E[Z_t]}^{2(p_2+1)}}{(p_2+1)C_2t}\right)\\
\rho_{Z_t}(y)\leq&\ \frac{\E[\abs{Z_t-\E[Z_t]}}{C_1t}\left(1+\abs{z}^{2p_2}\right)\exp\left(-\frac{2\abs{z_0-\E[Z_t]}^2}{C_2t}\left(1+z_0^{2p_2}\right)\right)\\
&\times\exp\left(-\frac{\abs{z-\E[Z_t]}^{p_1}-\abs{z_0-\E[Z_t]}^{p_1}}{p_1C_2t}\right).
\end{align*}
\end{itemize}
\end{Theorem}

\subsection{Verifying the assumptions of Theorem \ref{estim.dens}}

In this subsection, we give some conditions which ensure that the assumptions in Corollary \ref{cor_estime_y} hold. We recall that under Assumptions (X), (L) or (Q), (D1) and according to Proposition \ref{prop:Markov}, there exists a map $u:[0,T] \times \real \longrightarrow \real$ in $\mathcal{C}^{1,2}$ such that $Y_t =u(t,W_t), \ t\in [0,T], \ \P-$a.s., and $Z$ admits a continuous version given by $ Z_t = u'(t,W_t), \ t\in [0,T], \ \P-$a.s., assuming that $\sigma\equiv1$ and $b\equiv 0$ in the studied FBSDE \eqref{edsr}. Moreover we suppose for simplicity that the generator $h$ of BSDE \eqref{edsr} depends only on $z$, and that $u'$ and $u''$ are\footnote{This assumption is satisfied if $g$ and $h$ are smooth enough.} in $\mathcal{C}^{1,2}$. By a simple application of the non-linear Feynman-Kac formula (see for instance \cite{PardouxPeng_92}), and by differentiating it repeatedly, it can be shown that $u$, $u'$ and $u''$ are respectively classical solutions of the following PDEs:
 \begin{align}\label{PDE}
&\begin{cases}
\displaystyle -u_t(t,x)-\frac12 u_{xx}(t,x)-h(t, u_x(t,x))=0,\ (t,x)\in [0,T)\times \real  \\
\displaystyle u(T,x)=g(x),\ x\in \real,
\end{cases}\\
\label{PDE'}
 & \begin{cases}
\displaystyle - u'_t(t,x)-\frac12  u'_{xx}(t,x)-h_z(t, u'(t,x)) u'_x(t,x)=0,\ (t,x)\in [0,T)\times \real  \\
\displaystyle u'(T,x)=g'(x),\ x\in \real,\end{cases}\\
\label{PDE''}
 & \begin{cases}
\displaystyle - u''_t(t,x)-\frac12 u''_{xx}(t,x)-h_z(t, u'(t,x)) u''_x(t,x)-h_{zz}(t,u'(t,x))|u''(t,x)|^2=0,\; \hspace{-0.2em}(t,x)\in [0,T)\times \real  \\
\displaystyle u''(T,x)=g''(x),\ x\in \real.
\end{cases}
  \end{align}
  
We show in the following proposition and its corollary that under some conditions on $g,g',g''$ and $h, h_z$, the assumptions in Theorem \ref{estim.dens} are satisfied. We emphasize that this is only one possible set of assumptions, and that the required properties of $u$ and its derivatives can be checked on a case by case analysis.
  \begin{Proposition}\label{prop_illustration_lipy}
 Let $u$, $u'$ and $u''$ be respectively the solution to \reff{PDE}, \reff{PDE'} and \reff{PDE''} and assume that a comparison theorem holds for classical super and sub-solutions of these PDEs, in the class of functions with polynomial growth. Assume that there exist $(\varepsilon,\underline C,\overline C)\in(0,1)\times(0,+\infty)^3$, such that for all $x\in \real$
 $$\underline{C}(1+|x|^{1-\varepsilon})\leq g(x)\leq \overline{C}(1+|x|^{1+\varepsilon}).$$ 
Assume moreover that $h$ is non-positive and that there exist $(\varepsilon',\underline{D},\overline{D})\in(0,\varepsilon)\times(0,+\infty)^2$ s.t.
$$\underline{D}(1+|x|^{\varepsilon'})\leq g'(x)\leq \overline{D}(1+|x|^{\varepsilon}).$$

Assume that there exist  $(\underline{B},\overline{B}) \in (0,+\infty)^2$ such that for all $x\in \real$
$$\underline{B}\leq g''(x)\leq \overline{B}, \text{ and } 0\leq h_{zz}(t,x)<\frac{1}{4\overline{B}T}.$$
Assume finally that there exist $\lambda\in(0,\varepsilon^{-1}-1]$ and $C>0$ such that $|h_z(t,z)|\leq C(1+|z|^\lambda)$, then for all $(t,x)\in[0,T]\times\R$,
$$\underline{\alpha_u}\in [1-\varepsilon,1+\varepsilon], \ \overline{\alpha_{u'}},\underline{\alpha_{u'}}\in [\varepsilon',\varepsilon],\ \overline{\alpha_{u''}}=0, \ u'(t,x)\geq \underline D \text{ and } u''(t,x)\geq \underline{B}.$$

  \end{Proposition}
  
\begin{proof} 
 Let $\varphi(t,x):=\tilde{C}(T-t)+\overline{C}k_\varepsilon(x)$, where $k_\varepsilon(x)$ is in $\mathcal C^\infty(\R)$, coincides with the function $(1+\abs{x}^{1+\varepsilon})$ outside some closed interval centered at $0$ and is always greater than $(1+\abs{x}^{1+\varepsilon})$. We show that $\varphi$ is a (classical) super-solution to \reff{PDE} for some positive constant $\tilde{C}$ large enough. Indeed we can choose $\tilde{C}>0$ such that for any $(t,x)\in[0,T)\times\R$
$$ -\varphi_t(t,x)-\frac12 \varphi_{xx}(t,x)-h(t, \varphi_x(t,x))=\tilde{C}-\frac12 \overline{C}k_\varepsilon''(x)-h(t,\varphi_x(t,x))\geq 0,$$
since $h\leq 0$ and $\lim\limits_{|x|\to \infty }\frac12 k_\varepsilon''(x)=0$. 

\vspace{0.3em}
Moreover, by the assumption made on $g$, we clearly have for all $x\in\R$, $g(x)\leq \overline{C}k_\varepsilon(x)$, so that we deduce by comparison that for all $(t,x)\in [0,T] \times \real$:
$$ u(t,x)\leq \overline{C}k_\varepsilon(x)+\tilde{C}(T-t).$$
Now, we let $\phi(t,x):=-\tilde{C}_1(T-t)+\underline{C}\kappa_\varepsilon(x)$ for $(t,x)\in [0,T)\times \real$, where $\kappa_\varepsilon(x)$ is in $\mathcal C^\infty(\R)$, coincides with the function $(1+\abs{x}^{1-\varepsilon})$ outside some closed interval centered at $0$ and is always smaller than $(1+\abs{x}^{1-\varepsilon})$. We show that $\phi$ is a classical subsolution to \reff{PDE} for some positive constant $\tilde{C}_1$ large enough. We have
\begin{equation}\label{truc} - \phi_t(t,x)-\frac12 \phi_{xx}(t,x)-h(t,\phi_x(t,x))=-\tilde{C}_1+\frac12\underline{C}\kappa_\varepsilon''(x)-h(t, \phi_x(t,x)).
\end{equation}
Given that the quantity $h(t,\phi_x(t,x))=h(t,\underline C\kappa_\varepsilon'(x))$ is bounded because $\lim\limits_{|x|\to \infty}\kappa_\varepsilon'(x)=0$ and $h$ is continuous, we can always choose $\tilde C_1$ so that \reff{truc} is non-positive. Then, since we clearly have for all $x\in\R$, $g(x)\geq \underline{C}\kappa_\varepsilon(x)$, we deduce by comparison that for all $(t,x)\in [0,T] \times \real$:
$$ u(t,x)\geq \overline{C}\kappa_\varepsilon(x)+\tilde{C}_1(T-t).$$
 To sum up, we have showed that for all $(t,x)\in [0,T]\times \real$:
$$ \underline{C}\kappa_\varepsilon(x)-\tilde{C}_1(T-t)\leq u(t,x)\leq  \overline{C}k_\varepsilon(x)+\tilde{C}(T-t).$$ In other words $[\underline{\alpha_u},\overline{\alpha_u}]\subset [1-\varepsilon,1+\varepsilon]. $ 

\vspace{0.3em}
We  now study \reff{PDE'}. Define for some constant $\tilde C_2>0$ to be fixed later
\begin{align*}
\psi(t,x):=\tilde C_2(T-t)+\overline D\Upsilon_\varepsilon(x),
\end{align*}
where $\Upsilon_\varepsilon(x)$ is in $\mathcal C^\infty(\R)$, coincides with the function $(1+\abs{x}^{\varepsilon})$ outside some closed interval centered at $0$ and is always greater than $(1+\abs{x}^{\varepsilon})$. We then have
\begin{align*}
  -\psi_t(t,x)-\frac12 \psi_{xx}(t,x)-h_z(t,\psi(t,x)) \psi_x(t,x)=\tilde{C}_2-\frac12\overline{D}\Upsilon_\varepsilon''(x)-h_z(t,\psi(t,x))\overline D\Upsilon_\varepsilon'(x).
\end{align*}
Next, for some constant $C>0$ which may vary from line to line
$$|h_z(t,\psi(t,x))|\leq C(1+\abs{\psi(t,x)}^{\lambda})\leq C(1+\abs{x}^{\lambda\varepsilon}),$$ and since $\lambda \leq\frac1\varepsilon-1$ we deduce that:
$$\abs{h_z(t,\psi(t,x))\overline D\Upsilon_\varepsilon'(x)}\leq C(1+|x|^{\lambda \varepsilon+\varepsilon-1}) \text{, which is bounded. }$$  
Since in addition we have $\Upsilon_\varepsilon''(x)\longrightarrow 0$ as $\abs{x}$ goes to $+\infty$, we can always choose $\tilde C_2$ large enough so that
$$- \psi_t(t,x)-\frac12 \psi_{xx}(t,x)-h_z(t,\psi(t,x))\psi_x(t,x)\geq 0.$$
By the assumption we made on $g$, we can use once more the comparison theorem to obtain
$$u'(t,x)\leq \psi(t,x).$$ Similarly, we show that $\underline{D}\Upsilon_{\varepsilon'}(x)-\tilde{C}_3(T-t)$ is a sub-solution of \reff{PDE'} for some positive constant $\tilde{C}_3$, since $\lambda\leq \varepsilon^{-1}-1\leq \varepsilon'^{-1}-1$. Then, by comparison, we deduce that $\overline{\alpha_{u'}}, \underline{\alpha_{u'}} \in [\varepsilon', \varepsilon]$. Moreover, we notice that $\underline{D}\leq g'(x)$ for all $x\in \real$, so $\underline{D}$ is a sub-solution of \reff{PDE'}. Thus, using once more the comparison theorem $u'(t,x)\geq \underline{D}$ for all $(t,x)\in [0,T]\times \real$.

\vspace{0.4em}
We  now study \reff{PDE''}. Given that $h_{zz}$ is non negative and $\underline{B}\leq g''(x)$ for all $(t,x)\in [0,T]\times \real$, we deduce directly that $\underline{B}$ is a sub-solution of \reff{PDE''}. Next, let $\varpi(t,x)=\overline{B}+\frac{\overline{B}}{T^{1-\eta}}(T-t)^{1-\eta}$ where $\eta \in (0,1)$ is chosen small enough so that $h_{zz}(t,x)\leq \frac{1-\eta}{4T\overline{B}}$. Thus,
\begin{align*}
& - \varpi_t(t,x)-\frac12 \varpi_{xx}(t,x)-h_z(t, u'(t,x)) \varpi_x(t,x)-h_{zz}(t,u'(t,x))|\varpi(t,x)|^2\\
&=(1-\eta)\frac{\overline{B}}{T^{1-\eta}}(T-t)^{-\eta}-h_{zz}(t,u'(t,x))\overline{B}^2\left(1+\frac{(T-t)^{1-\eta}}{T^{1-\eta}} \right)^2\\
&\geq (1-\eta)\frac{\overline{B}}{T^{1-\eta}}(T-t)^{-\eta}-\frac{1-\eta}{4T}\overline{B}\left(1+\frac{(T-t)^{1-\eta}}{T^{1-\eta}} \right)^2\\
&\geq 0.
 \end{align*}
 
We deduce that $\varpi$ is a super solution of \reff{PDE''}, which by comparison, implies that $u''$ is bounded, so $\overline{\alpha_{u''}}=0$.
\end{proof}
  
\begin{Corollary}
  
Consider the FBSDE \eqref{edsr} and assume that for all $t\in [0,T]$ $X_t=W_t$ and $h$ depends only on $z$. Let $u(t,X_t):=Y_t$ and assume that $u\in \Cc^{1,2}$, $u'\in \Cc^{1,2}$ and $u''\in \Cc^{1,2}$. Let the assumptions of Proposition \ref{prop_illustration_lipy} hold, and assume moreover that $\varepsilon \in (0,\frac12)$. Then, the assumptions of Theorem \ref{estim.dens} hold.
\end{Corollary}
  
 \begin{proof}
  According to Proposition \ref{prop_illustration_lipy}, $\underline{\alpha_u}\geq 1-\varepsilon$, $\overline{\alpha_{u'}}\leq \varepsilon$ and $u'(t,x)\geq \underline D,\ (t,x)\in[0,T]\times\R$. From the fact that $\varepsilon$ is smaller than $1/2$, we deduce that $0\leq\overline{\alpha_{u'}}<\underline{\alpha_{u}}<+\infty$. Moreover, $0=\overline{\alpha_{u''}}<\varepsilon'\leq\underline{\alpha_{u'}}$.
\end{proof}

\begin{flushright}
\vspace{-2em}
$\qed$
\end{flushright}

\section*{Acknowledgments}
Thibaut Mastrolia is grateful to R\'egion Ile-De-France for financial support. The authors thank an Associate Editor and two anonymous Referees for their careful reading of this paper and for insightful suggestions which have greatly improve its presentation.

\newpage
\section{Table of assumptions-results}

In this appendix we recall the different assumptions made within this paper and we give a summary table of some most significant results on BSDEs including ours.

\paragraph*{Assumption for $X$:}

\begin{itemize}
\item[(X)]$b,\sigma : [0,T]\times \real \longrightarrow \real$ are continuous in time and continuously differentiable in space for any fixed time $t$ and such that there exist $k_b,k_\sigma >0$ with 
$$|b_x(t,x)|\leq k_b,\ |\sigma_x(t,x)|\leq k_\sigma, \text{ for all $x\in\R$}.$$ 
Besides $b(t,0), \sigma(t,0)$ are bounded functions of $t$ and there exists $c>0$ such that for all $t\in [0,T]$ $$0<c\leq |\sigma(t,\cdot)|, \ \lambda(dx)-a.e.$$
\end{itemize}

\textbf{List of assumptions for BSDEs:} 

\begin{itemize}
\item[(L)] 
\begin{itemize}
\item[(i)] $g : \mathbb{R} \longrightarrow \mathbb{R}$ is such that $\E[g(X_T)^2]<+\infty$.
\item[(ii)]  $h : [0,T]\times \mathbb{R}^3 \longrightarrow \mathbb{R}$ is such that there exist $(k_x,k_y,k_z)\in(\R_+^*)^3$ such that for all $(t,x_1,x_2,y_1,y_2,z_1,z_2) \in [0,T]\times \mathbb{R}^6$, 
$$ |h(t,x_1,y_1,z_1)-h(t,x_2,y_2,z_2)|\leq k_x|x_1-x_2|+k_y|y_1-y_2|+k_z|z_1-z_2|.$$
\item[(iii)] $\int_0^T |h(s,0,0,0)|^2ds<+\infty$.
\end{itemize}
\item[(Q)]
\begin{itemize}
\item[(i)] $g : \mathbb{R} \longrightarrow \mathbb{R}$ is bounded.
\item[(ii)] $h : [0,T]\times \mathbb{R}^3 \longrightarrow \mathbb{R}$ is such that:
\begin{itemize}
\item[$\triangleright$] There exists $(K,K_z,K_y)\in(\R_+^*)^3$ such that for all $(t,x,y,z) \in [0,T]\times \mathbb{R}^3$
$$\hspace{-3em}|h(t,x,y,z)|\leq K(1+|y|+|z|^2),\ \abs{h_z}(t,x,y,z)\leq K_z(1+|z|),\ \abs{h_y}(t,x,y,z)\leq K_y.$$
\item[$\triangleright$]  There exists $C>0$ such that for all $(t,x,y,z_1,z_2) \in [0,T] \times \real^4$
\begin{equation*} |h(t,x,y,z_1)-h(t,x,y,z_2)| \leq C(1+|z_1|+|z_2|) |z_1-z_2|.\end{equation*}
\end{itemize}
\item[(iii)] $\int_0^T |h(s,0,0,0)|^2ds<+\infty$.
\end{itemize}

\end{itemize}

\paragraph*{List of assumptions for Malliavin differentiability of $(X,Y,Z)$:} 

\begin{itemize}
\item[(D1)] 

\begin{itemize}
\item[(i)] $g$ is differentiable, $\mathcal{L}(X_T)-$a.e., $g$ and $g'$ have polynomial growth.
\item[(ii)] $(x,y,z)\mapsto h(t,x,y,z)$ is continuously differentiable for every $t$ in $[0,T]$.

\end{itemize}
\item[(D2)] 
\begin{itemize}
\item[(i)] $g$ is twice differentiable, $\mathcal{L}(X_T)-$a.e., $g$, $g'$ and $g''$ have polynomial growth. 

\item[(ii)] $(x,y,z)\mapsto h(t,x,y,z)$ is twice continuously differentiable for every $t$ in $[0,T]$.
\end{itemize}
\end{itemize}

\paragraph*{List of assumptions for the existence of densities for $Y$ and $Z$:}

$$\underline{g}:= \inf\limits_{x \in \real} g'(x), \quad \underline{g}^A:=\inf\limits_{x\in A} g'(x),\quad \overline{g}:= \sup\limits_{x \in \real} g'(x), \quad \overline{g}^A:=\sup\limits_{x\in A} g'(x),$$
$$\underline{h}(t):=\inf\limits_{ s\in [t,T], (x,y,z) \in \real^3} h_x(s,x,y,z),\quad \quad \overline{h}(t):=\sup\limits_{ s\in [t,T], (x,y,z) \in \real^3} h_x(s,x,y,z),$$

and $K:=k_b+k_y+k_{\sigma}k_z$. There exists $A\in\mathcal B(\R)$ such that $\mathbb{P}(X_T \in A | \mathcal{F}_t)>0$ and such that:
\begin{align*}
 &(H+)\quad \begin{cases}
\displaystyle \underline{g}e^{-\sgn(\underline{g})KT}+\underline{h}(t)\int_t^T e^{-\sgn(\underline{h}(s))Ks}ds\geq0 \\
\displaystyle \underline{g}^Ae^{-\sgn(\underline{g}^A)KT}+\underline{h}(t)\int_t^T e^{-\sgn(\underline{h}(s))Ks}ds>0
\end{cases}\\[0.3em]
&(H-)\quad \begin{cases}
\displaystyle  \overline{g}e^{-\sgn(\overline{g})KT}+\overline{h}(t)\int_t^T e^{-\sgn(\overline{h}(s))Ks}ds\leq 0 \\
\displaystyle \overline{g}^Ae^{-\sgn(\overline{g}^A)KT}+\overline{h}(t)\int_t^T e^{-\sgn(\overline{h}(s))Ks}ds<0,
\end{cases}
\end{align*}

Set
\begin{align*}
\tilde{h}(s,x,y,z):=& -\left( h_{xt}+b h_{xx}-hh_{xy}+\frac12(\sigma^2 h_{xxx}+2z\sigma h_{xxy}+z^2h_{xxy})\right)(s,x,y)\nonumber\\
&-\left((h_y+b_x)h_x+\sigma \sigma_xh_{xx}+z\sigma_xh_{xy}\right)
(s,x,y).\\
 \tilde{g}(x):=&\ g'(x)+(T-t)h_x(T,x,g(x)),
 \end{align*}
 and 
$$ \underline{\tilde{g}}:=\min\limits_{x\in \mathbb{R}} \tilde{g}(x), \quad \overline{\tilde{g}}:=\max\limits_{x\in \mathbb{R}} \tilde{g}(x),\quad \underline{\tilde{g}}^A:=\min\limits_{x\in A} \tilde{g}(x), \quad \overline{\tilde{g}}^A:=\max\limits_{x\in A} \tilde{g}(x),$$
$$\underline{\tilde{h}}(t):=\min\limits_{[t,T]\times \mathbb{R}^3} \tilde{h}(s,x,y,z), \ \overline{\tilde{h}}(t):=\max\limits_{[t,T]\times \mathbb{R}^3} \tilde{h}(s,x,y,z),$$
and set $K:=k_y+k_b$. There exists $A\in\mathcal B(\R)$ such that $\mathbb{P}(X_T \in A | \mathcal{F}_t)>0$  \begin{align*}
&\widetilde{(H+)} \quad  \begin{cases}
\displaystyle \underline{\tilde{g}}e^{-\sgn(\underline{\tilde{g}})KT}+\underline{\tilde{h}}(t)\int_t^T e^{-\sgn(\underline{\tilde{h}}(s))Ks}(T-s)ds\geq0 \\
\displaystyle \underline{\tilde{g}}^Ae^{-\sgn(\underline{\tilde{g}}^A)KT}+\underline{\tilde{h}}(t)\int_t^T e^{-\sgn(\underline{\tilde{h}}(s))Ks}(T-s)ds>0,
\end{cases}\\[0.3em]
&\widetilde{(H-)} \quad \begin{cases}
\displaystyle \overline{\tilde{g}}e^{-\sgn(\overline{\tilde{g}})KT}+\overline{\tilde{h}}(t)\int_t^T e^{-\sgn(\overline{\tilde{h}}(s))Ks}(T-s)ds\leq 0 \\
\displaystyle \overline{\tilde{g}}^Ae^{-\sgn(\overline{\tilde{g}}^A)KT}+\overline{\tilde{h}}(t)\int_t^T e^{-\sgn(\overline{\tilde{h}}(s))Ks}(T-s)ds<0.
\end{cases}
\end{align*}

\begin{itemize}
\item[$(Q+)$] $g'\geq 0$ and  $g'_{\vert A} >0$, $\mathcal{L}(X_T)-$a.e. and $\underline{h}(t)\geq 0$,
\item[$(Q-)$] $g'\leq 0,\ g'_{\vert A}<0$, $\mathcal{L}(X_T)-$a.e. and $\overline{h}(t)\leq 0$,
\item[(Z+)]
\begin{itemize}
\item There exist $(\underline a,\overline a)$ s.t., $0<\underline{a}\leq  D_r X_u \leq \overline{a}$, for all $0<r<u\leq T$.
\item There exists $\overline b$ s.t., $0\leq  D_{r,s}^2 X_u\leq \overline{b}$, for all $0<r,s<u\leq T$.
\item $h_x,h_{xx},h_{yy},h_{zz},h_{xy}\geq 0$ and $h_{xz}= h_{yz}= 0$ (and $h_y\geq 0$ under $(Q)$)
\item $h_{xy}= 0$ or $(h_{xy}\geq 0$ and $g'\geq 0$, $\mathcal{L}(X_T)$-a.e.$)$.
\item We have $$ \mathbf{1}_{\{\underline{g''}<0\}}\underline{g''}\overline{a}^2+\underline{g'}\mathbf{1}_{\{\underline{g'}<0\}}\overline{b}+(\mathbf{1}_{\{\underline{g''}\geq 0\}}\underline{g''}+\underline{h_{xx}}(t)(T-t))\underline{a}^2\geq 0,$$
and
$$ (\mathbf{1}_{\{\underline{g''}^A<0\}}\underline{g''}^A\overline{a}^2+\underline{g'}^A\mathbf{1}_{\{\underline{g'}<0\}}\overline{b}) +(\mathbf{1}_{\{\underline{g''}^A\geq 0\}}\underline{g''}^A+\underline{h_{xx}}(t)(T-t))\underline{a}^2>0,$$
\end{itemize}
\end{itemize}

We give the following summary table which sums up significant results for BSDEs in both the Lipschitz case and the quadratic case with assumptions made and references. 
\small
\begin{center}
\begin{tabular}{|c|c|c|}
 \hline \diagbox{Results}{Cases} & Lipschitz case (L) &  Quadratic case (Q)  \\
 \hline Existence and uniqueness  &\multirow{2}{*}{Prop. \ref{propex} (X)}  & \multirow{2}{*}{Prop. \ref{propexq} (X) }\\
 of solutions of BSDEs&  &  \\
\hline Malliavin differentiability & \multirow{2}{*}{Prop. \ref{MD}  (X) and (D1) } &  \multirow{2}{*}{Prop. \ref{MDq} (X) and (D1)}  \\
of $(X,Y,Z)$  &  &\\
\hline \multirow{2}{*}{Density existence for $Y$} & Th. \ref{thm_H+H-} (X), (D1) and (H+) or (H-)& \multirow{2}{*}{Th. \ref{thm_H+H-_quadra}  (X), (D2) and (Q+) or (Q-)} \\
&  Th. \ref{AKmodifie} (X), (D1) and ($\widetilde{H+}$) or ($\widetilde{H-}$) &\\
\hline \multirow{2}{*}{Density existence for $Z$} & \multirow{2}{*}{Th. \ref{thm_density_z_lip} (X), (D2) and (Z+)} & \multirow{2}{*}{Th. \ref{thm_density_z_quadra}  (X), (D2) and (Z+)}\\
& &\\
\hline 
\end{tabular}
\end{center}

\end{document}